\documentclass{hha}

\usepackage{amsmath,amsbsy,amsfonts,amssymb,
stmaryrd,amsthm,mathrsfs, amscd, tikz-cd}
\usepackage{dsfont}
\usepackage[all]{xy}
\usepackage{mathtools}
\usepackage{subcaption}
\usepackage{dutchcal}
\usepackage[numbers,sort&compress]{natbib}
\usepackage{xcolor}
\colorlet{darkgreen}{green!40!black}
\usepackage{ifpdf}
\ifpdf \usepackage[pdftex]{hyperref}
\else \usepackage[ps2pdf]{hyperref} 
      \usepackage{breakurl}
\fi
\hypersetup{colorlinks=true,urlcolor=blue,citecolor=darkgreen,
linkcolor=darkgreen,linktocpage,bookmarksnumbered,unicode}

\newcommand\Sh{\operatorname{Sh}}

\newcommand\Hom{\operatorname{Hom}}

\newcommand\Img{\operatorname{Im}}

\newcommand\Id{\operatorname{Id}}

\newcommand\Dig{\operatorname{Dig}}
\newcommand\hDig{\operatorname{hDig}}
\newcommand\BDig{\operatorname{Dig^{*}}}
\newcommand\hBDig{\operatorname{hDig^{*}}}

\newcommand\Acal{\mathcal A}

\newcommand\Qcal{\mathcal Q}
\newcommand\Pcal{\mathcal P}

\newcommand\Kbb{\mathbb{K}}

\newcommand\Rbb{\mathbb{R}}
\newcommand\Zbb{\mathbb{Z}}

\makeatletter

\newcommand{\Rnum}[1]{\expandafter\@slowromancap\romannumeral #1@}

\begin{document}

\title{Algebra of Iterated Path Integrals on Digraphs}


\author{Shing-Tung Yau}

\email{styau@tsinghua.edu.cn}

\address{Yau Mathematical Sciences Center, Tsinghua
University, Beijing 100084, China\\
Beijing Institute of Mathematical Sciences and Applications, 101408, China}


\author{Mengmeng Zhang}
\email{mengmengzhang@bimsa.cn}
\address{Beijing Key Laboratory of Topological Statistics and Applications for Complex Systems, Beijing Institute of Mathematical Sciences and Applications, Beijing 101408, China}
\thanks{The second author was supported by the Start-up Research Fund at BIMSA, Beijing Postdoctoral Research Foundation, the Interdisciplinary Fund of SIMIS.}
\author{Yunpeng Zi}
\email{ziypmath\_2024@email.sdu.edu.cn}
\address{{School of Mathematics}, Shandong University, Jinan, Shandong 250100, China}
\thanks{The third author was supported by the Natural Science Foundation of Shandong Province China Grant ZR2025QC1476, {the National Natural Science Foundation of China Grant 12601125}}

\classification{05C20, 16T05, 53A70, 55P35.}

\keywords{Iterated Path Integral, Digraph, Fundamental Group, Hopf Algebra}

\begin{abstract}
In this paper, we extend the iterated path integrals from smooth manifolds to digraphs and develop the associated algebraic and geometric structures. Iterated path integrals on a digraph naturally give rise to the iterated path algebra and the iterated loop algebra, both defined as quotient algebras of a shuffle algebra, with the latter carrying a canonical Hopf algebra structure.

We construct a non-degenerate pairing between elementarily equivalent classes of loops on a digraph and the iterated loop algebra. By restricting to iterated path integrals that are invariant under $C_\partial$-homotopy, a distinguished subalgebra is obtained which, under this pairing, corresponds to the group algebra of the fundamental group. We further show that this subalgebra is a homotopy invariant and forms a Hopf algebra with involutive antipode.

\end{abstract}


\received{Month Day, Year}   
\revised{Month Day, Year}    
\published{Month Day, Year}  
\submitted{Henry Adams}      
\volumeyear{*} 
\volumenumber{*} 
\issuenumber{*}   
\startpage{1}     
\articlenumber{*} 

\owner{International Press}

\maketitle


\section{Introduction}
Directed graphs (digraphs) serve as fundamental models for numerous complex systems in the real world.
{A framework for the homology of digraphs was introduced by Grigor'yan, Lin, Muranov, and Yau~\cite{grigor2012homologies}, providing effective tools for detecting and analyzing connectivity patterns and cycles in networks, as well as for studying complex disease systems, molecular isomers, and related applications~\cite{chen2023path,chowdhury2019path,wu2023metabolomic}.} 
{Subsequently, Asao, Hepworth, and others
introduced several homology theories that generalize path homology
\cite{asao2023magnitude, caputi2024hochshild, hepworthbigraded, hepworth2025reachability}, while Muranov, Tang et al. developed
other related homological theories~\cite{Yuricubical, Tangcelluar}. Furthermore, the homotopy theory of digraphs~\cite{grigor2014homotopy} also introduced by Grigor'yan et al., and was subsequently developed from both digraph and cubical set viewpoints~\cite{carranza2024cubical, li2024homotopy, kishimoto2025fundamental,sergeihomotopy}. These
developments have attracted attention from mathematicians
and researchers in related interdisciplinary fields, contributing to
the development of a new area of study known as GLMY theory~\cite{chen2023path,carranza2024cofibration, chowdhury2019path,chowdhurypersis,chowdhury2022path,di2024path,fu2024path,li2024homotopy,wu2023metabolomic}.}

In 2012, Grigor'yan, Lin, Muranov, and Yau defined the path homology and cohomology theories using paths and differential forms on digraphs, proving that both are homotopy invariants. Subsequently, in 2014, they defined the fundamental group of a based digraph as the $C_{\partial}$-homotopy classes of loops and developed a corresponding homotopy theory for digraphs~\cite{grigor2014homotopy}. In 2015, Grigor'yan et al. further provided a combinatorial definition of equivalence for the fundamental group and proved that its abelianization is isomorphic to the first path homology group~\cite{grigor2018fundamental}. The cohomology theory of digraphs was also further studied in~\cite{grigor2015cohomology}.
{These developments of cohomology and homotopy theories on digraphs naturally motivate us to seek a discrete analogue of Chen's theory of iterated path integrals on digraphs.
}

The iterated path integral is a classical tool in calculus that reduces multiple integrals to single integrals. In 1957, K.T. Chen introduced iterated path integrals on the path space of a smooth manifold in \cite{chen1957integration} and \cite{chen1958integration}. He subsequently constructed a Hopf algebra with an involutive antipode, generated by homotopy-invariant iterated path integrals, and showed that this algebra is homotopy invariant and dual to the fundamental group $\pi_1$ \cite{chen1971algebras}. Furthermore, the iterated path integral and homotopy theory were deeply explored by K.T. Chen, who called it de Rham homotopy theory, establishing a connection between de Rham theory, the loop space of a differentiable space, and the homotopy groups. A more recent work by R. Hain pointed out in \cite{hain1987geometry} and \cite{hain1985iterated} that the iterated path integral leads to a mixed Hodge structure on the homotopy group of a smooth complex algebraic variety.

Inspired by de Rham homotopy theory, this paper aims to introduce a new homotopy invariant for {finite} digraphs, based on the cochain complex and fundamental group of digraphs developed by Grigor'yan \emph{et al.} First, we generalize the notion of iterated path integrals to finite path maps $\alpha$ and $1$-forms $\omega_1,\ldots,\omega_r$ on digraphs, and denote such an integral by
\begin{equation*}
    \int_{\alpha}\omega_1\cdots\omega_r.
\end{equation*}
Collecting all iterated path integrals yields an algebra whose multiplication is given by the shuffle product. Moreover, the iterated path integral along the inverse of a path map equals the iterated path integral taken in reverse order, up to a sign $\pm 1$.
To this end, we introduce the notions of indefinite iterated path integrals and indefinite iterated loop integrals by allowing the finite path map $\alpha$ to remain undetermined. The algebras $\Pcal(G)$ and $\Qcal(G)$ are then generated by all such indefinite iterated path integrals and loop integrals, respectively, with unit $1$ under the shuffle multiplication.
They are quotients of the shuffle algebra. Moreover, $\Qcal(G)$ is a Hopf algebra with {an involutive antipode}, and the proof depends on the following well-defined bilinear pairing.

Fix a vertex $x \in V(G)$. Let $P_x(G)$ (resp.\ $L_x(G)$) denote the path space (resp.\ loop space) of a digraph $G$, defined as the $\mathbb{K}$-vector space generated by all path maps starting at $x$ (resp.\ all loop maps based at $x$).

There exists a natural bilinear pairing between the iterated path algebra $\mathcal{P}_x(G)$, which obtained as the restriction of $\mathcal{P}(G)$ at $x$, and the path space $P_x(G)$: 
\begin{equation}\label{eq:bilinear-form-ori}
\begin{aligned}
(-,-): \mathcal{P}_x(G)\times P_x(G) &\longrightarrow \mathbb{K}, \\
\left(\int_x \omega_1 \cdots \omega_r,\; \alpha\right) &\longmapsto \int_\alpha \omega_1 \cdots \omega_r.
\end{aligned}
\end{equation}

Moreover, this pairing naturally restricts to a bilinear pairing between the iterated loop algebra $\mathcal{Q}_x(G)$, regarded as a subalgebra of $\mathcal{P}_x(G)$, and the loop space $L_x(G)$:
\begin{equation}\label{eq:bilinear-form-loop}
\begin{aligned}
(-,-): \mathcal{Q}_x(G)\times L_x(G) &\longrightarrow \mathbb{K}, \\
\left(\int_x \omega_1 \cdots \omega_r,\; \alpha\right) &\longmapsto \oint_{\alpha} \omega_1 \cdots \omega_r.
\end{aligned}
\end{equation}

We show that the kernels of these pairings consist precisely of those path maps that are elementarily equivalent to the trivial path. Here, the equivalence relation is generated by two elementary moves: collapsing trivial arrows and inserting or removing backtracking arrows along finite paths.

In particular, we obtain the following characterization theorem.
\begin{theorem}
{Let $\alpha,\beta \in P_{x,y}(G)$}. Two path maps $\alpha$ and $\beta$ are elementarily equivalent if and only if
\[
\int_{\alpha} \omega_1 \cdots \omega_r
=
\int_{\beta} \omega_1 \cdots \omega_r
\]
for all $1$-forms $\omega_1,\ldots,\omega_r$ with $r \geq 1$.
\end{theorem}

However, iterated path integrals are not well-defined on the homotopy classes of paths, as their values depend on the choice of a specific representative. To address this issue, we isolate those elements that are invariant under $C_{\partial}$-homotopy class of path maps and thereby construct a new algebra $\pi^{1}(G)$ associated to a digraph $G$. This yields a nontrivial involutive Hopf algebra structure, which depends only on the homotopy type of the digraph.
\begin{theorem}{}
Let $\Kbb$ be a field with characteristic 0, $G^{*}=(G_0,G_1,x)$ be a based digraph and $\omega_1,\cdots,\omega_r$ be 1-forms on $G^{\ast}$. Define
\begin{itemize}
    \item $\Pcal_x(G)$ be the {$\Kbb$-algebra} generated by the iterated path integrals $\int_x\omega_1\cdots\omega_r$,
    \item $\Qcal_x(G)$ be the {$\Kbb$-algebra} generated by the iterated loop integrals $\oint_x\omega_1\cdots\omega_r$,
    \item $\pi^1_x(G)$ be the sub-$\Kbb$-algebra of $\Qcal_x(G)$ consisting of all iterated loop integrals that are independent of {$C_{\partial}$-homotopy} of loops.
\end{itemize}
 There exist {contra}variant functors:
    \begin{equation*}
    \begin{aligned}
    \Pcal\colon  \underline{\BDig}\to \underline{\mathcal{A}};&\quad G^*=(G_0,G_1,x)\mapsto \Pcal_x(G)\\
    \Qcal\colon  \underline{\BDig}\to \underline{\mathcal{H}};&\quad G^*=(G_0,G_1,x)\mapsto \Qcal_x(G)\\
          \pi^1\colon  \underline{\hBDig}\to \underline{\mathcal{H}};&\quad G^*=(G_0,G_1,x)\mapsto\pi^1_x(G),
    \end{aligned}
     \end{equation*}
     where $\underline{\BDig}$ is the category of based digraphs, $\underline{\hBDig}$ is its homotopy category, and $\underline{\mathcal{A}}$ and $\underline{\mathcal{H}}$ {are the categories} of $\mathbb{K}$-algebras and Hopf $\mathbb{K}$-algebras, respectively.
\end{theorem}
{This theorem shows that $\pi^1(G)$ is a homotopy invariant of any digraph $G$ as a Hopf $\Kbb$-algebra.}

Finally, we emphasize that the constructions and results in this paper are not restricted to digraphs. They extend naturally to quivers without self-loops.

This paper is organized as {following}: we review the differential forms and homotopy theory of digraphs in Section 2. Then in Section 3, we introduce the iterated path integrals of a digraph and discuss {their} properties, where we also explore the elementary equivalence on iterated path integrals. Subsequently, several Hopf algebras {of} iterated path integrals on digraphs are studied in Section 4. Finally we define a homotopy invariant $\pi^1$ in Section 5 as a Hopf sub-algebra of $\Qcal$.
\section{Preliminaries}
In this section, we review some basic definitions and notations about  differential forms on {finite} digraphs~\cite{grigor2015cohomology} and digraph homotopy theory~\cite{grigor2014homotopy} introduced by Grigor'yan, Lin, Muranov and Yau. {And throughout this paper, all digraphs are assumed to be finite}. 

A \textit{digraph} $G$ is a pair $(G_0, G_1)$, where $G_0$ is the set of vertices and
$G_1 \subset (G_0 \times G_0) \setminus \operatorname{diag}(G_0)$ is the set of arrows, with $\operatorname{diag}(G_0) = \{(x, x) \mid x \in G_0\}$.
An arrow $(x, y) \in G_1$ is typically denoted by $x \to y$. {In this paper, we assume that digraphs have no self-loops.}  Moreover, the \textit{source} and \textit{target} maps $s, t \colon G_1 \to G_0$ are defined by
\[
s(x, y) = x, \qquad t(x, y) = y,
\]
respectively. A \emph{based digraph} $G^*$ is a pair $(G, x_0)$, where $G$ is a digraph and $x_0 \in G_0$ is the basepoint.

An important example is the \textit{line digraph}. A \textit{line digraph} $I_n$ of length $n$ is a digraph with vertex set $(I_n)_0 = \{0,1,\cdots,n\}$, and for each $0 \leq i \leq n-1$, there is exactly one arrow between $i$ and $i+1$, either $i \to i+1$ or $i \leftarrow i+1$. For example, $I_0$ consists of a single vertex, and $I_1$ consists of a single arrow.
The associated based line digraph is denoted by $I_n^* = (I_n, 0)$. Moreover, the inverse $\widehat{I}_n$ of the line digraph $I_n$ is defined by
\begin{equation*}
	i \to j \in (\widehat{I}_n)_1 \iff (n-i) \to (n-j) \in (I_n)_1.
\end{equation*}

Let $G, H$ be two digraphs. A \textit{digraph map} $f \colon   G \to H$ is a map $f \colon   G_0 \to H_0$ such that for any arrow $(x, y) \in G_1$, $(f(x), f(y)) \in H_1 \cup \operatorname{diag}(H_0)$.
And for based digraphs $G^{*}=(G,g_0)$ and $H^{*}=(H,h_0)$, a \textit{based digraph map} is a digraph map $f\colon   G\to H$ such that $f(g_0) =h_0$.
In particular, a digraph map $\alpha\colon   I_n\rightarrow G$ from line digraph $ I_n$ to digraph $G$ is called a \textit{path map}. The based path map $\alpha^{\ast}\colon   I_n^{*}\rightarrow G^{\ast}$ is called a path map from $g_0$, and called a loop if in addition $\alpha^{\ast}(0)=\alpha^{\ast}(n)$. {And it is worth noting that, although our digraphs do not contain self-loops, the image $(f(x),f(y))$ of $(x,y)\in G$ may lie in $\operatorname{diag}(H_0)$ for a digraph map $f:G\to H$. For later convenience, we refer to elements $(z,z)\in \operatorname{diag}(H_0)$ as \textit{trivial arrows} and denote them by $r_z$.
}

Several operations {are} defined on finite path maps. The \textit{inverse} $\alpha^{-1}\colon   \widehat{I}_n\to G $ of path map $\alpha$ is a digraph map defined by $\alpha^{-1}(i) = \alpha(n-i)$.  Clearly, the inverse of a loop is still a loop. {For example, the inverse of an arrow $a=x\to y$ is $a^{-1}=y\to x$. }
Given two path maps $\alpha\colon  I_n\to G$ and $\alpha'\colon  I_m\to G$, the \textit{concatenation} $\alpha\star\alpha'$ of $\alpha$ and $\alpha'$ is defined by
\begin{equation*}
\alpha \star \alpha'(i)=
\begin{cases}
0, & \text{if } \alpha(n) \neq \alpha'(0),\\[4pt]
\alpha(i), & \text{if } \alpha(n) = \alpha'(0)\ \text{and}\ 0\leq i \leq n, \\
\alpha'(i-n), & \text{if }\alpha(n) = \alpha'(0)\ \text{and}\ n+1\leq i \leq m+n.\\
\end{cases}
\end{equation*}
The \textit{cut at $[a,b]$} of $\alpha\colon  I_n\to G$, denoted by $\alpha|_{[a,b]}$, is the path map $I_{b-a}\to G$ such that $\alpha|_{[a,b]}(t)=\alpha(t+a)$. For simplicity, $\alpha|_{[a,n]}$ is written as $\alpha|_{[a:]}$.

The \textit{box product} $G\Box H$ of $G$ and $H$ is a digraph defined as {follows}: the vertices are the elements of $G_0\times H_0$, the arrow $(x_1,x'_1)\to (x_2,x'_2)$ lies in $(G\Box  H)_1$ if and only if $x_1=x_2$ and $x'_1\to x'_2\in H_1$, or $x_1\to x_2\in G_1$ and $x'_1=x'_2$. Let  $f\colon  G\to H$ be a digraph map, the \textit{direct cylinder} of $f$ is the digraph  $C_f$ with $(C_f)_0=G_0\sqcup H_0$ and $(C_f)_1=G_1\sqcup H_1\sqcup \{v\to f(v)\colon  v\in G_0\}$. The \textit{inverse cylinder} of $f$ is the digraph  $C^{-}_f$ with $(C^{-}_f)_0=G_0\sqcup H_0$ and $(C^{-}_f)_1=G_1\sqcup H_1\sqcup \{f(v)\to v:v\in G_0\}$.

\subsection{Differential Forms on Digraphs}
Let $G=(G_0,G_1)$ be a digraph, let $\Kbb$ be a fixed field of characteristic $0$, and let $n$ be a non-negative integer. An \textit{elementary $n$-path} on vertex set $G_0$ is a sequence of vertices $v_0 v_1 \cdots v_n$, where $v_i \in G_0$ for $0 \leq i \leq n$.

Denote by $\Lambda_n(G_0)$ the $\Kbb$-vector space generated by all elementary $n$-paths on $G_0$, that is,
\[
\Lambda_n(G_0)
:= \operatorname{span}_{\Kbb}
\bigl\{ e_{v_0 v_1 \cdots v_n} \mid
v_0 v_1 \cdots v_n \text{ is an elementary $n$-path on } G_0 \bigr\}.
\]

Then $\Lambda_{\ast}(G_0)$, together with the boundary homomorphism
\begin{equation*}
\partial_n \colon \Lambda_n(G_0)\to \Lambda_{n-1}(G_0),\quad
a e_{v_0v_1\cdots v_{n}} \mapsto a \sum_{i=0}^{n}(-1)^i e_{v_0v_1\cdots v_{i-1}\widehat{v_i}v_{i+1}\cdots v_n}
\end{equation*}
forms a chain complex. An elementary $n$-path $v_0 v_1 \cdots v_n$ is called \textit{regular} if $v_i \neq v_{i+1}$ for all $0 \leq i \leq n-1$. The free $\Kbb$-module generated by all regular $n$-paths is denoted by $R_n(G_0)$. Otherwise, it is called a non-regular $n$-path, and the free $\Kbb$-module generated by all non-regular $n$-paths is denoted by $I_n(G_0)$. It follows that $R_n(G_0)$ is isomorphic to $\Lambda_n(G_0)/I_n(G_0)$.

Taking the arrow information of $G$ into account, the \textit{allowed $n$-path} on $G$ is defined to be a vertex sequence $v_0v_1\cdots v_n$ such that $(v_i\to v_{i+1})\in G_1$ for each $0\leq i\leq n-1$. Otherwise, it is called non-allowed. In contrast to path maps, allowed $n$-paths do not involve trivial arrows, and all arrows are of the form $v_i\to v_{i+1}\in G_1$ for all $0\leq i\leq n-1$.

Then we write
\begin{equation*}
    \Acal_n(G):=\{e_{v_0v_1\cdots v_n}\mid v_0v_1\cdots v_n \text{ is an allowed $n$-path on } G\}
\end{equation*}
as a submodule of $\Lambda_n(G_0)/I_n(G_0)$. However, $\Acal_n(G)$ is not closed under $\partial_n$. Yau et al.\ introduced the $\partial$-invariant chain complex $(\Omega_{\ast}(G),\partial_{\ast})$, where
\begin{equation*}
    \Omega_{n}(G)=\Acal_n(G)\cap \partial_n^{-1}\Acal_{n-1}(G)
\end{equation*}
for $n\geq 0$. It follows that
$R_0(G_0)=\Acal_0(G)=\Omega_0(G)=\Kbb G_0$
and
$\Acal_1(G)=\Omega_1(G)=\Kbb G_1$.

Dual to $\Omega_n(G)$, $\Omega^n(G)$ is considered as the $\Kbb$-module of $n$-forms on $G$. Define \[ \Lambda^n(G_0):=\Kbb\{e^{v_0v_1\cdots v_n}\mid v_0v_1\cdots v_n \text{ is an elementary $n$-path on } G_0\}, \]
where $e^{v_0 v_1\cdots v_n}\colon \Lambda_n(G_0)\to \Kbb$ is given by
\begin{equation*}
    e^{v_0v_1\cdots v_n}(w_0w_1\cdots w_n)=
    \begin{cases}
        1, & \text{if } v_0v_1\cdots v_n=w_0w_1\cdots w_n,\\
        0, & \text{otherwise}.
    \end{cases}
\end{equation*}
With this notation, the following submodules of $\Lambda^n(G_0)$ are introduced:
\begin{equation*}
\begin{aligned}
R^n(G_0)&:=\Kbb\{e^{v_0v_1\cdots v_n}\mid v_0v_1\cdots v_n \text{ is a regular $n$-path on } G_0\},\\
\Acal^n(G)&:=\Kbb\{e^{v_0v_1\cdots v_n}\mid v_0v_1\cdots v_n \text{ is an allowed $n$-path on } G\},\\
N^n(G)&:=\Kbb\{e^{v_0v_1\cdots v_n}\mid v_0v_1\cdots v_n \text{ is a non-allowed $n$-path on } G\}.
\end{aligned}
\end{equation*}
And the coboundary homomorphism $\delta^n\colon \Lambda^n(G_0)\to \Lambda^{n+1}(G_0)$ is defined by $\delta^n(f)=f\circ \partial_{n+1}$. At the same time, define
\begin{equation*}
\begin{aligned}
J^n(G)&:=N^n(G)+\delta^{n-1}(N^{n-1}(G)),\\
\Omega^n(G)&:=\Acal^n(G)/\bigl(\Acal^n(G)\cap J^n(G)\bigr).
\end{aligned}
\end{equation*}

An element $f\in \Omega^n(G)$ is called an $n$-form on $G$. By \cite[Lemma 3.19]{grigor2012homologies}, $\Omega^n(G)$ is identified with the dual space of $\Omega_n(G)$.
\begin{remark}
A path map, or an $n$-path $\alpha$ on $G$, may also be represented by a sequence of arrows $a_1 a_2 \cdots a_n$, where $a_i$ denotes the arrow $\alpha(i-1)\to \alpha(i)$ in $G$.
Accordingly, an $n$-form on $G$ is written as $e^{a_1 a_2 \cdots a_n}$.
Thus, every $n$-form $\omega$ is a $\Kbb$-linear combination of elements of the form $e^{a_1 a_2 \cdots a_n}$.
These equivalent representations of path maps, $n$-paths, and $n$-forms are used interchangeably in what follows, depending on the context.

Moreover, a path map, \(n\)-path, or \(n\)-form \(\alpha\) on \(G\) determines a sequence of vertices \(\alpha(i)\) for all \(i\). Conversely, a given sequence of vertices may correspond to several distinct path maps, and thus to different \(n\)-paths or \(n\)-forms. Nevertheless, such vertex sequences will be used later to describe local loop transformations and to study the fundamental group of a digraph.
\end{remark}

At the end of this subsection, several algebraic properties of $\Omega^p(G)$ are recalled.
Let $\Acal(G):=\Lambda^0(G_0)=R^0(G_0)=\Acal^0(G)$ be the space of $0$-forms on $G$.
It is a unital $\Kbb$-algebra by definition.
For any $p\geq 0$, the space of $p$-forms $\Omega^p(G)$ is {an} $\Acal(G)$-bimodule, the point multiplications are defined as following:
\begin{equation*}
\begin{aligned}
    e^{v}\cdot e^{a}=\begin{cases}
                    0, v\neq s(a),\\
                    e^{a}, v=s(a)
	                   \end{cases}\qquad \qquad
    e^{a}\cdot e^{v}=\begin{cases}
                    0, v\neq t(a),\\
                    e^{a}, v=t(a).
	                   \end{cases}
\end{aligned}
\end{equation*}

In general, let $f=\sum_{v\in G_0} f(v)e^v$ and $\omega=\sum_{a\in G_1}\omega(a)e^a$. Then the two multiplications are given by
\begin{equation*}
\begin{aligned}
f\cdot \omega &= \sum_{a\in G_1} f(s(a))\,\omega(a)\,e^a,\\
\omega\cdot f &= \sum_{a\in G_1} f(t(a))\,\omega(a)\,e^a.
\end{aligned}
\end{equation*}
Therefore, $(f\cdot\omega)(a)=f(s(a))\omega(a)$ and $(\omega\cdot f)(a)=f(t(a))\omega(a)$.

\subsection{Homotopy Theory on Digraphs} \label{subsec:pre-homotopy-dig}

Let $f,g\colon  G\rightarrow G'$ be two digraph maps, $f$ is called \textit{homotopic} to $g$, denoted by $f\simeq g$, if there is a line digraph $I_n$ and a digraph map $F\colon   G\,\Box \, I_n\rightarrow G'$ such that $F|_{G\Box \{0\}} = f$, $F|_{G\Box \{n\}} = g$. Two digraphs $G$ and $G'$ are homotopy equivalent if there exist digraph maps $f\colon G\to G'$ and $g\colon G'\to G$ such that
$f\circ g\simeq \Id_{G'}$ and $g\circ f\simeq \Id_{G}$.
In this case, we write $G\simeq G'$, and $f$ and $g$ are called homotopy inverses of each other.

For example, there are three standard digraphs that are homotopic to a single vertex, as shown in Figure~\ref{fig:contrible-dig}.
\begin{figure}[h]
\centering
\begin{subfigure}{0.3\textwidth}
	\begin{equation*}
		\xymatrix{v_0\ar[r]^{a_1}\ar[rd]^{a_3}&v_1\ar[d]^{a_2}\\&v_2}
	\end{equation*}
	\caption{Standard Triangle}
\end{subfigure}
\hfill
\begin{subfigure}{0.3\textwidth}
	\begin{equation*}
		\xymatrix{v_0\ar[r]^{a_1}\ar[d]^{a_3}&v_1\ar[d]^{a_2}\\v_2\ar[r]^{a_4}&v_3}
	\end{equation*}
	\caption{Standard Square}
\end{subfigure}
\hfill
\begin{subfigure}{0.3\textwidth}
	\begin{equation*}
		\xymatrix{v_0\ar@/^/[r]^{a_1}&v_1\ar@/^/[l]^{a_2}}
	\end{equation*}
	\caption{Standard Double Edge}
\end{subfigure}
\caption{Three contractible digraphs.}
\label{fig:contrible-dig}
\end{figure}
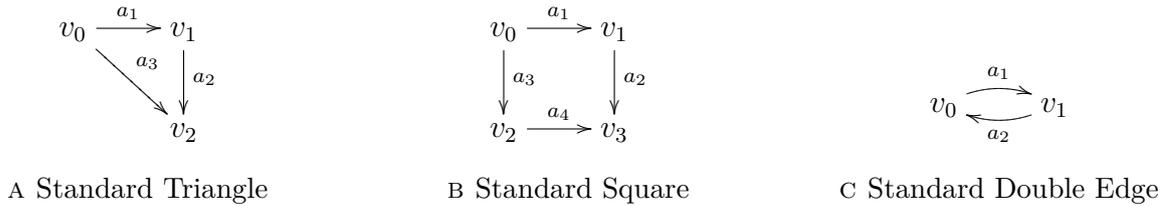
Further, a triple of vertices $v_0, v_1, v_2$ is called a \emph{triangle} if there exists a permutation $\sigma$ of $\{v_0, v_1, v_2\}$ such that $\sigma(v_0), \sigma(v_1), \sigma(v_2)$ form a standard triangle.
Similarly, four vertices $v_0, v_1, v_2, v_3$ are called a \emph{square} if there exists a cyclic permutation $\sigma$ of $\{v_0, v_1, v_2, v_3\}$ such that $\sigma(v_0), \sigma(v_1), \sigma(v_2), \sigma(v_3)$ form a standard square.

Let $\underline{\Dig}$ and $\underline{\BDig}$ denote the categories of digraphs and based digraphs, respectively.
Their objects are digraphs and based digraphs, and their morphisms are digraph maps and based digraph maps, respectively.
The corresponding homotopy categories are denoted by $\underline{\hDig}$ and $\underline{\hBDig}$, whose objects are as above and whose morphisms are homotopy classes of digraph maps and based digraph maps, respectively.

The homotopy theory of path maps and based path maps on a given digraph $G$ is now considered.
Suppose $\alpha\colon I_n\to G$ and $\beta\colon I_m\to G$ are two path maps.
Without loss of generality, assume $m\leq n$.
A \textit{shrinking map} is a digraph map $h\colon I_n\to I_m$ such that $h(0)=0$, $h(n)=m$, and $h(i)\leq h(j)$ whenever $i\leq j$.

A path map $\alpha$ is said to be \textit{one-step directly $C_{\partial}$-homotopic} to $\beta$, denoted by $\alpha\simeq_1 \beta$, if there exists a digraph map $F\colon C_h\to G$ such that
\begin{equation*}
    F|_{I_n}=\alpha, \qquad F|_{I_m}=\beta.
\end{equation*}
If the same holds with \(C_h\) replaced by \(C_h^{-}\), we say that \(\alpha\) is \emph{one-step inverse \(C_{\partial}\)-homotopic} to \(\beta\). Both one-step directly and one-step inverse \(C_{\partial}\)-homotopies are referred to as \emph{one-step \(C_{\partial}\)-homotopy}.
Moreover, if there exists a finite sequence of path maps \(\{\varphi_k\}_{k=0}^m\) with \(\varphi_0 = \alpha\) and \(\varphi_m = \beta\), such that for each \(k = 0, \dots, m-1\), either \(\varphi_k\) is one-step \(C_{\partial}\)-homotopic to \(\varphi_{k+1}\) or \(\varphi_{k+1}\) is one-step \(C_{\partial}\)-homotopic to \(\varphi_k\), then \(\alpha\) and \(\beta\) are said to be \emph{\(C_{\partial}\)-homotopic}, denoted \(\alpha \simeq_C \beta\).

The $C_{\partial}$-homotopy equivalence of path maps admits a combinatorial description as follows~\cite{grigor2018fundamental}.
\begin{theorem}
Two path maps $\alpha\colon I_n \rightarrow G$ and $\beta\colon I_m \rightarrow G$ are said to be \emph{$C_{\partial}$-homotopy equivalent} if $\alpha$ can be obtained from $\beta$ by a finite sequence of the following local transformations of vertex sequences and their inverses:
\begin{enumerate}
    \item[(i)] $\cdots v_0v_1v_2 \cdots \;\mapsto\; \cdots v_0v_2 \cdots$
    if $v_0, v_1, v_2$ form a triangle in $G$;

    \item[(ii)] $\cdots v_0v_1v_3 \cdots \;\mapsto\;
          \cdots v_0v_2v_3 \cdots$
    if $v_0, v_1, v_2, v_3$ form a square in $G$;

    \item[(iii)] $\cdots v_0v_1v_3v_2 \cdots \;\mapsto\;
          \cdots v_0v_2 \cdots$
    if $v_0, v_1, v_2, v_3$ form a square in $G$;

    \item[(iv)] $\cdots v_0v_1v_0 \cdots \;\mapsto\; \cdots v_0v_0 \cdots$
    if $v_0 \to v_1$, or $v_1 \to v_0$, or $v_0 = v_1$;

    \item[(v)] $\cdots v_0v_0v_1 \cdots \;\mapsto\; \cdots v_0v_1 \cdots$.
\end{enumerate}
Here the dots ``$\cdots$'' denote the unchanged parts of the vertex sequences.
\end{theorem}
\section{Iterated Path Integrals on Digraphs}
Let $\alpha\colon I_n\to G$ be a path map on $G$, and let $\alpha(i)$ denote its $i$-th vertex. Define the associated arrow $\alpha_i$ as follows: if the $i$-th arrow is of the form $\alpha(i-1)\to \alpha(i)$, then $\alpha_i$ is taken to be this arrow; otherwise, it is taken to be its inverse. In particular, if $\alpha(i-1)\to \alpha(i)$ is a trivial arrow, $\alpha_i$ is itself.

In this setting, there is a well-defined pairing
\[
\langle-,-\rangle\colon \Omega^1(G)\times \Lambda_1(G_0)\to \mathbb{K},
\]
given by $(\omega,a)\mapsto \omega(a)$, and extended linearly by
\[
\langle \omega,a^{-1}\rangle = -\omega(a), \qquad a\in G_1,
\]
and $\langle \omega,b\rangle=0$ for any trivial arrow $b$.

Let $I=(t_1,\cdots,t_r)$ be a non-decreasing sequence of positive integers. Then there {exist} $m$ integers
$1\leq i_1<\cdots<i_m=r$ such that
\[
t_1=\cdots=t_{i_1}<t_{i_1+1}=\cdots=t_{i_2}<\cdots<t_{i_{m-1}+1}=\cdots=t_{i_m}.
\]

\begin{definition}
The \emph{volume number} $\tau(t_1,\cdots,t_r)$ of the sequence $(t_1,\cdots,t_r)$ is defined by
\begin{equation*}
\tau(t_1,\cdots,t_r)=i_1!(i_2-i_1)!\cdots(i_m-i_{m-1})!.
\end{equation*}
\end{definition}

\begin{remark}
Let $\varphi_k(t_1,\dots,t_r)$ denote the multiplicity of $k$ in $(t_1,\dots,t_r)$ for $k\geq 1$. Then
\begin{equation*}
\tau(t_1,\cdots,t_r)=\prod_{k\geq 1}\varphi_k(t_1,\cdots,t_r)!.
\end{equation*}

{The geometry beyond the {coefficient} $\frac{1}{\tau(t_1,\cdots,t_r)}$ is as following: for each non-decreasing sequence $(t_1,\cdots,t_r)$, there exists a cube $C(t_1,\cdots,t_r)$ determined by $(t_1,\cdots,t_r)$ in $\Rbb^r$ namely
\begin{equation*}
	C(t_1,\cdots,t_r):=\{(x_1,\cdots,x_r): 0\leq t_k-x_k\leq 1,\forall k=1,\cdots,r\}.
\end{equation*}
Then $\frac{1}{\tau(t_1,\cdots,t_r)}$ is the volume for the subspace of $C(t_1,\cdots,t_r)$ satisfying $0\leq x_1\leq \cdots\leq x_r$, see Figure \ref{fig:vol-number}.}
\begin{figure}[htbp]
\centering
\begin{subfigure}{0.5\textwidth}
\includegraphics[width=\linewidth]{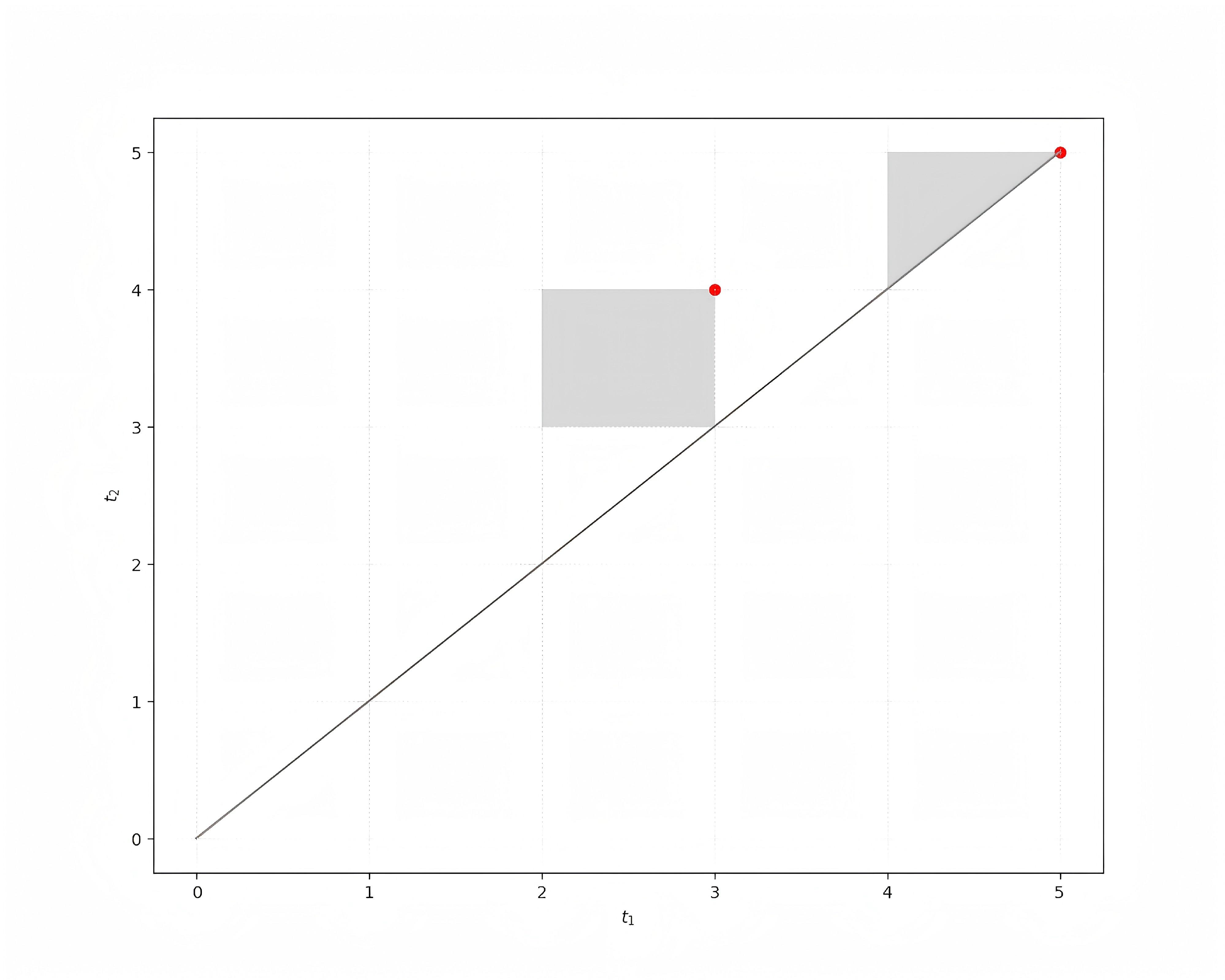}
\end{subfigure}
\hfill
\caption{{Visualization for $r=2$. For the point $(t_1,t_2)=(3,4)$, $\tau(3,4)=1$ and for the point $(t_1,t_2)=(5,5)$,  $\tau(5,5)=2$}}
\label{fig:vol-number}
\end{figure}
\end{remark}

\begin{definition}\label{iteratedintegral}
   Let $\omega_1,\cdots,\omega_r$ be 1-forms and $\alpha\colon I_n\to G$ be a path map on $G$. Then the \textit{iterated path integral} of $\omega_1,\cdots,\omega_r$ over $\alpha$ denoted by $\int_{\alpha}\omega_1\cdots\omega_r$ is defined by
   \begin{equation*}
	\begin{aligned}
	&\int_{\alpha}\omega_1\cdots\omega_r
	&=\sum_{1\leq t_1\leq t_2\leq \cdots\leq t_r\leq n}\frac{\langle\omega_1, \alpha_{t_1} \rangle\cdots\langle\omega_r,\alpha_{t_r}\rangle}{\tau(t_1,\cdots,t_r)}=\sum_{1\leq t_1\leq t_2\leq \cdots\leq t_r\leq n}\frac{\prod\limits_{i=1}^r\langle\omega_i, \alpha_{t_i} \rangle}{\tau(t_1,\cdots,t_r)}.
	\end{aligned}
\end{equation*}
If $\alpha$ is a loop, i.e. $\alpha(0)=\alpha(n)$, the iterated path integrals are denoted by  $\oint_{\alpha}\omega_1\cdots\omega_r$. In addition, we refer to $\int_{\alpha} 1 = 1$.
\end{definition}
There are several simple examples.
\begin{enumerate}
    \item if $r=1$, $\int_{\alpha}\omega=\sum\limits_{1\leq i\leq n}\omega(\alpha_i)$;
    \item if $r=2$, $\int_{\alpha}\omega_1\omega_2=\sum\limits_{1\leq t_1<t_2\leq n}\omega_1(\alpha_{t_1})\omega_2(\alpha_{t_2})+\sum\limits_{t=1}^n\frac{\omega_1(\alpha_t)\omega_2(\alpha_t)}{2}$;
    \item if $\alpha\colon I_1\to G$ is a path map, $\int_{\alpha} \omega_1\cdots\omega_r= \frac{\omega_1(\alpha_1)\cdots\omega_r(\alpha_1)}{r!}$;
    \item if $\alpha\colon I_2\to G$ is a two-step path map, $\int_{\alpha}\omega_1\cdots\omega_r=\sum\limits_{k=0}^r\frac{\omega_1(\alpha_1)\cdots\omega_k(\alpha_1)\omega_{k+1}(\alpha_2)\cdots\omega_r(\alpha_2)}{k!(r-k)!}$.
\end{enumerate}
{
\begin{remark} The examples above suggest a possible interpretation of
Definition~\ref{iteratedintegral} for directed networks. If a $1$-form
$\omega$ assigns a weight or loss to each arrow, then
$\int_{\alpha}\omega$
gives the total weight or loss along the path map $\alpha$. Moreover, $
\int_{\alpha}\omega_1\omega_2$
records the ordered interaction  with
$\omega_1$ occurring before $\omega_2$ along $\alpha$. For example,
$\omega_1$ may represent activation amplification and $\omega_2$
inhibition or delay. This
order-sensitive feature may be useful for describing information flow
in neural networks or multistage processes in cybersecurity attack
digraphs, where the ordering of interactions can be as important as
their accumulated magnitudes. Further investigation and potential applications deserve to be explored in specific
settings. 
\end{remark} }

It is clear that the iterated path integrals depend {not only} on the choice of differential forms but also on the choice of path maps. We may make the path map undetermined and define the indefinite iterated path integrals.
Let $x,y$ be two fixed vertices of $G$, denoted by $P(G)$, $P_x(G)$ and  $P_{x,y}(G)$ the $\Kbb$-vector space spanned by all {finite} path maps on $G$,  path maps from $x$ and  path maps from $x$ to $y$, respectively. Similarly, denote $L(G)$ and $L_x(G)$ as the $\Kbb$-vector space spanned by all finite loops on $G$ and  all loops based at $x$ on $G$, respectively.
\begin{definition}
Let $\omega_1,\cdots,\omega_r$ be 1-forms on digraph $G$. The indefinite iterated path integral on $G$  denoted by $\int \omega_1\cdots\omega_r$, the indefinite iterated path integral from $x$ on $G$ denoted by  $\int_x\omega_1\cdots\omega_r$ and the indefinite iterated path integral from $x$ to $y$  denoted by $\int_x^y\omega_1\cdots\omega_r$, are the {linear} {functionals} on path {space} $P(G)$, $P_x(G)$ or $P_{x,y}(G)$, respectively, given by
\begin{equation}\label{eq:dual-PG}
	\begin{aligned}
	\int\omega_1\cdots \omega_r\colon P(G)\to \Kbb, &  \quad \alpha\mapsto \int_{\alpha}\omega_1\cdots \omega_r, \\
    \int_x\omega_1\cdots \omega_r\colon P_x(G)\to \Kbb, &\quad \alpha\mapsto \int_{\alpha}\omega_1\cdots \omega_r, \\ \int_x^y\omega_1\cdots \omega_r\colon P_{x,y}(G)\to \Kbb, &\quad \alpha\mapsto  \int_{\alpha}\omega_1\cdots \omega_r.
	\end{aligned}
\end{equation}
Similarly the indefinite loop iterated path integral $\oint \omega_1\cdots \omega_r$, the  indefinite loop iterated path integral based at $x$ $\oint_x\omega_1\cdots \omega_r$ are {linear} {functionals} on $L(G)$, $L_x(G)$ respectively given by
\begin{equation}
	\begin{aligned}
	\oint\omega_1\cdots \omega_r\colon L(G)\to \Kbb, &\qquad \alpha\mapsto \oint_{\alpha}\omega_1\cdots \omega_r, \\
  \oint_x\omega_1\cdots \omega_r\colon L_x(G)\to \Kbb,  &\qquad \alpha\mapsto \oint_{\alpha}\omega_1\cdots \omega_r.
	\end{aligned}
\end{equation}
\end{definition}

\subsection{Basic Properties of Iterated Path Integrals}
We present several basic properties of iterated path integrals on digraphs in this subsection. For simplicity, throughout this and the following subsection, we assume that $\omega_1,\cdots,\omega_r$ are 1-forms on $G$.

\begin{proposition}\label{prop:ii-conca-path}
Let $\alpha\colon I_n\to G$ and $\beta\colon I_m\to G$ be two path maps such that $\alpha(n)=\beta(0)$. Then
\begin{equation*}
\int_{\alpha\star\beta}\omega_1\cdots\omega_r
=\sum_{i=0}^r \int_{\alpha}\omega_1\cdots\omega_i\,
\int_{\beta}\omega_{i+1}\cdots\omega_r.
\end{equation*}
\end{proposition}

\begin{proof}
By the definition of concatenation of path maps, one has
$(\alpha\star\beta)_t=\alpha_t$ for $1\leq t\leq n$, and
$(\alpha\star\beta)_t=\beta_{t-n}$ for $n+1\leq t\leq m+n$.

Let $(t_1,\cdots,t_r)$ be a non-decreasing sequence with $t_j\in\{1,2,\cdots,m+n\}$, ${1}\leq j\leq r$. Then there exists a unique index $0\leq i\leq r$ such that
\[
1\leq t_1 \leq \cdots \leq t_i \leq n,
\qquad
n+1\leq t_{i+1}\leq \cdots \leq t_r \leq m+n.
\]
Consequently, the volume number satisfies
\begin{equation*}
\tau(t_1,\cdots,t_r)
=
\tau(t_1,\cdots,t_i)\,\tau(t_{i+1},\cdots,t_r)
=
\tau(t_1,\cdots,t_i)\,\tau(t_{i+1}-n,\cdots,t_r-n).
\end{equation*}

It follows from the definition of iterated path integrals that
\begin{equation*}
\int_{\alpha\star\beta}\omega_1\cdots\omega_r
=
\sum_{1\leq t_1\leq\cdots\leq t_r\leq m+n}
\frac{\prod_{j=1}^r \langle \omega_j,(\alpha\star\beta)_{t_j}\rangle}
{\tau(t_1,\cdots,t_r)}.
\end{equation*}

On the other hand,
\begin{equation*}
\sum_{i=0}^r \int_{\alpha}\omega_1\cdots \omega_i\,
\int_{\beta}\omega_{i+1}\cdots \omega_r
=
\sum_{i=0}^r
\sum_{\substack{1\leq t_1\leq\cdots\leq t_i\leq n\\
n+1 \leq t_{i+1}\leq\cdots\leq t_r\leq m+n}}
\frac{\prod_{j=1}^i \langle\omega_j,\alpha_{t_j}\rangle}
{\tau(t_1,\cdots,t_i)}
\frac{\prod_{j=i+1}^r \langle\omega_j,\beta_{t_j-n}\rangle}
{\tau(t_{i+1}-n,\cdots,t_r-n)}.
\end{equation*}
Combining the above identities yields the desired equality.
\end{proof}

Let $\alpha\colon I_n \to G$ be a path map. {Subdividing $\alpha$ at vertex $i$, we obtain two path maps
\[
\widehat{\alpha}^i \colon I_{n+1}^i \to G
\quad \text{and} \quad
\widehat{\alpha}^{-i} \colon I_{n+1}^{-i} \to G,
\]
for $0 \leq i \leq n$ such that $\widehat{\alpha}^i (i) = \widehat{\alpha}^i (i+1)$ for arrow $i\to i+1 \in I_{n+1}^i$ and $\widehat{\alpha}^{-i} (i) = \widehat{\alpha}^{-i} (i+1)$ for arrow $i\to i+1 \in I_{n+1}^{-i}$.} These two path maps have different domains but the same image in $G$. The following corollary shows that the iterated path integral is invariant under this operation.

\begin{corollary}\label{trivial}
Let $\alpha\colon I_n \to G$ be a path map.
\[
\int_{\widehat{\alpha}^i}\omega_1\cdots\omega_r
=
\int_{\widehat{\alpha}^{-i}}\omega_1\cdots\omega_r
=
\int_{\alpha}\omega_1\cdots\omega_r.
\]
\end{corollary}

\begin{proposition}\label{prop:ii-inv-path}
Let $\alpha\colon I_n\to G$ a path map on $G$.
	\begin{equation*}
		\int_{\alpha^{-1}}\omega_1\cdots\omega_r=(-1)^r\int_{\alpha}\omega_r\cdots\omega_1.
	\end{equation*}
\end{proposition}
\begin{proof}
From the definition of inverse path map, $(\alpha^{-1})_i=\alpha_{n+1-i}^{-1}$. Therefore, {there are properties} $\langle\omega_i,(\alpha^{-1})_j\rangle=-\langle\omega_i,\alpha_{n+1-j}\rangle$. Moreover the definition of volume number implies
\begin{equation*}
\tau(t_1,\cdots,t_r)=\tau(n+1-t_r,\cdots,n+1-t_1).
\end{equation*}
Hence there are
\begin{equation*}
	\begin{aligned}
&\int\limits_{\alpha^{-1}}\omega_1\cdots\omega_r=\sum_{\substack{1\leq t_1\leq\cdots\leq t_r\leq n}}\frac{\prod\limits_{i=1}^r\langle{\omega_i},(\alpha^{-1})_{t_i}\rangle}{\tau(t_1,\cdots,t_r)}\\
	=(-1)^r&\int\limits_{\alpha}\omega_r\cdots\omega_1=(-1)^r\sum_{\substack{1\leq t_1\leq\cdots\leq t_r\leq n}}\frac{\prod\limits_{{i=1}}^r\langle\omega_i,\alpha_{n+1-t_i}\rangle}{\tau(t_1,\cdots,t_r)}.
	\end{aligned}
\end{equation*}
\end{proof}

Recall that a permutation $\sigma$ of $(r+s)$ letters is called a $(r,s)$-shuffle if and only if $\sigma^{-1}(1)< \cdots<\sigma^{-1}(r)$ and $\sigma^{-1}(r+1)<\cdots<\sigma^{-1}(r+s)$.

\begin{proposition}\label{prop:ii-mult}
	Let $\omega_1,\cdots,\omega_r,\cdots,\omega_{r+s}$ be 1-forms on $G$.
	\begin{equation}\label{eq:ii-mult}
	\int_{\alpha}\omega_1\cdots\omega_r\int_{\alpha}\omega_{r+1}\cdots\omega_{r+s}=\sum_{\sigma \in Sh(r,s)}\int_{\alpha}\omega_{\sigma (1)}\cdots\omega_{\sigma(r+s)},
    \end{equation}
    where $\sigma$ {runs} over all $(r,s)$-shuffles.
\end{proposition}
\begin{remark}
The multiplication in this proposition is referred to the \emph{shuffle multiplication}, namely
\begin{equation*}
\omega_1\cdots\omega_r \circ \omega_{r+1}\cdots\omega_{r+s}
:=
\sum_{\substack{\sigma^{-1}(1)< \cdots<\sigma^{-1}(r)\\
\sigma^{-1}(r+1)<\cdots<\sigma^{-1}(r+s)}}
\omega_{\sigma(1)}\cdots\omega_{\sigma(r+s)}.
\end{equation*}
With this notation, the above identity can be rewritten as
\begin{equation*}
\int_{\alpha}\omega_1\cdots\omega_r\;
\int_{\alpha}\omega_{r+1}\cdots\omega_{r+s}
=
\int_{\alpha}\omega_1\cdots\omega_r \circ \omega_{r+1}\cdots\omega_{r+s}.
\end{equation*}
\end{remark}
\begin{proof}
The proof is obtained by comparing the coefficients of the term
\begin{equation}\label{eq:given-term}
\langle\omega_1,\alpha_{t_1}\rangle\cdots\langle\omega_{r+s},\alpha_{t_{r+s}}\rangle.
\end{equation}

Let $I=(t_1,\cdots,t_{r+s})$ be a non-decreasing sequence with $1\leq t_i\leq n$. Then $I$ is the union of two non-decreasing subsequences
$I'=(t_1\leq\cdots\leq t_r)$ and $I''=(t_{r+1}\leq\cdots\leq t_{r+s})$.
Clearly, $
\varphi_i(I)=\varphi_i(I')+\varphi_i(I'')$ for $1 \leq i$.

The left-hand side of \eqref{eq:ii-mult} is given by
\begin{equation*}
	\begin{aligned}
\int_{\alpha}\omega_1\cdots\omega_r\int_{\alpha}\omega_{r+1}\cdots\omega_{r+s}
=\sum_{\substack{1\leq t_1\cdots\leq t_r\leq n\\1\leq t_{r+1}\leq\cdots t_{r+s}\leq n}}\frac{\prod\limits_{i=1}^r\langle \omega_i,\alpha_{t_i}\rangle\prod\limits_{i=r+1}^{r+s}\langle \omega_{i},\alpha_{t_{i}}\rangle}{\tau(t_1,\cdots,t_r)\tau(t_{r+1},\cdots,t_{r+s})}.
	\end{aligned}
\end{equation*}
Hence the coefficient of \eqref{eq:given-term} on the left-hand side equals
\[
\frac{1}{\prod_{i\geq 1}\varphi_i(I')!\,\varphi_i(I'')!}.
\]

The right-hand side of \eqref{eq:ii-mult} can be written as
\begin{equation*}
	\begin{aligned}
\sum_{\sigma}\int_{\alpha}\omega_{\sigma 1}\cdots\omega_{\sigma(r+s)}
=\sum_{1\leq t_1\leq\cdots\leq t_{r+s}\leq n}\sum_{\substack{\sigma \in Sh(r,s)}}\frac{\prod\limits_{i=1}^{r+s}\langle \omega_i,\alpha_{t_{\sigma^{-1}(i)}}\rangle}{\tau(t_1,\cdots,t_{r+s})}.
	\end{aligned}
\end{equation*}
The permutations of $(r+s)$-letters that stabilize the sequence $I$ form a group $S$ with totality
\begin{equation*}
	\Pi_{i=1}^n(\varphi_i(I')+\varphi_i(I''))!.
\end{equation*}
Observe that the group $S$ {exchanges} the numbers with equal value between $I'$ and $I''$. Clearly, the stabilizer of $I'$ and $I''$ has totality
\begin{equation*}
	\Pi_{i=1}^n\varphi_i(I')!\varphi_i(I'')!.
\end{equation*}
Then the $(r,s)$-shuffles {correspond} to the orbits of $S$-action on $I$. Hence the number of the $(r,s)$-shuffles that stabilize $I$ is
\begin{equation*}
	\Pi_{i=1}^n\frac{(\varphi_i(I')+\varphi_i(I''))!}{\varphi_i(I')!\varphi_i(I'')!}.
\end{equation*}
{Note that $\tau(I)=\Pi_{i=1}^n(\varphi_i(I')+\varphi_i(I'')){!}$}. Therefore the coefficients on the right side {are} the following
\begin{equation*}
	\Pi_{i=1}^n\frac{1}{(\varphi_i(I')+\varphi_i(I''))!}\cdot \frac{(\varphi_i(I')+\varphi_i(I''))!}{\varphi_i(I')!\varphi_i(I'')!}=\Pi_{i=1}^n\frac{1}{\varphi_i(I')!\varphi_i(I'')!}.
\end{equation*}
\end{proof}
Assume that $f=\sum_{v\in G_0}f(v)e^v \in \Acal(G)$. The differential of $0$-form was defined in \cite{grigor2015cohomology} that $de^v=\sum_{\substack{a\in G_1\\  t(a)=v}}e^a-\sum_{\substack{a\in G_1\\  s(a)=v}}e^a$, therefore
\begin{equation}\label{eq:df}
	df=\sum_{v\in G_0}(\sum_{\substack{a\in G_1\\  t(a)=v}}f(v)e^a-\sum_{\substack{a\in G_1\\  s(a)=v}}f(v)e^a)=\sum_{a\in G_1}(f({t(a)})-f({s(a)}))e^a.
\end{equation}
The following Lemma holds.

\begin{lemma}\label{lem:int-df}
Let $f\in \Acal(G)$ be any function and $\alpha\colon I_n\to G$ be a path map
on $G$.
\begin{equation}
	\int_{\alpha}df=f(\alpha(n))-f(\alpha(0)).
\end{equation}
In particular, if $\alpha$ is a loop then $\int_{\alpha}df=0$.
\end{lemma}
\begin{proof}
{From Equation (\ref{eq:df}), it is clear that
\begin{equation*}
	\int_{\alpha}df=\sum_{t=1}^n\sum_{a\in G_1}(f({t(a)})-f({s(a)}))e^a(\alpha_t)=f(\alpha(n))-f(\alpha(0)).
\end{equation*}
The last assertion follows naturally.}
\end{proof}

\begin{proposition}\label{prop:iter-mul-func}
Let $c\colon  \Lambda_0(G) \to \Kbb$  be a constant function valued at $k\in \Kbb$.
    \begin{equation*}
        \begin{aligned}
            &\int_{\alpha}\omega_1\cdots\omega_{r-1} (\omega_r\cdot c)\omega_{r+1}\cdots\omega_{m}
		=k\int_{\alpha}\omega_1\cdots\omega_{r-1}\omega_r\omega_{r+1}\cdots\omega_{m},\\
		&\int_{\alpha}\omega_1\cdots\omega_{r-1} (c\cdot\omega_r)\omega_{r+1}\cdots\omega_{m}
        =k\int_{\alpha}\omega_1\cdots\omega_{r-1}\omega_r\omega_{r+1}\cdots\omega_{m}.
        \end{aligned}
    \end{equation*}
\end{proposition}
\begin{proof}
The proof is clear.
\end{proof}
\begin{remark}
This highlights a difference between iterated path integrals on digraphs and those on smooth manifolds. In the manifold setting, the corresponding identity holds for any smooth map~\cite{chen1968algebraic}, \cite{chen1971algebras}, i.e.
\begin{equation*}
		\begin{aligned}
		&\int_{\alpha}\omega_1\cdots\omega_{r-1} (\omega_r\cdot f)\omega_{r+1}\cdots\omega_{m}\\
		=&f(\alpha(0))\int_{\alpha}\omega_1\cdots\omega_{r-1}\omega_r\omega_{r+1}\cdots\omega_{m}
+\int_{\alpha}(\omega_1\cdots\omega_{r-1}\circ df)\omega_r\omega_{r+1}\cdots\omega_{m}.
		\end{aligned}
\end{equation*}
However, the corresponding analogue for digraphs remains unclear in general.
\end{remark}
\subsection{Elementary Equivalence}
\begin{definition}
Two path maps $\alpha$ and $\beta$ are said to be \emph{one-step elementarily equivalent} if one can be transformed into the other by applying the following two operations:
\begin{enumerate}
	\item $\gamma'\star\gamma\star\gamma^{-1}\star\gamma''\sim \gamma'\star\gamma''$, where $\gamma$, $\gamma'$ and $\gamma''$ are allowed to be trivial path maps;
	\item $\gamma'\star r \star\gamma''\sim \gamma'\star\gamma''$, where $r$ is a trivial arrow.
\end{enumerate}
\end{definition}

Two path maps are called \emph{elementarily equivalent} if one can be obtained from the other by a finite sequence of one-step elementary equivalences.

A \emph{reduced path map} $\beta \colon I_n \to G$ on $G$ is a path map such that for every $1\leq i \leq n$, neither $\beta_i$ nor $\beta_j \star \beta_{j+1}$ is elementarily equivalent to a trivial path map.

The following proposition shows that iterated path integrals are invariant under elementary equivalence.
\begin{proposition}\label{prop:iter-ele-equ}
Let $\alpha$, $\beta$ and $\gamma$ be path maps on digraph $G$. For any 1-forms $\omega_1,\cdots,\omega_r$, $r\geq 1$, $\int_{\beta\star\alpha\star\alpha^{-1}\star\gamma}\omega_1\cdots\omega_r=\int_{\beta\star\gamma}\omega_1\cdots\omega_r$. In particular, $\int_{\alpha\star\alpha^{-1}}\omega_1\cdots\omega_r=0$.
\end{proposition}
\begin{proof}
From Proposition \ref{prop:ii-conca-path}, it is sufficient to show the special case where $\beta$ resp. $\gamma$ are trivial path {maps}. Assume that $\alpha\colon  I_n\to G$ and $\rho=\alpha\star\alpha^{-1}$. Then $\rho$ is the path map $I_{2n}\to G$. Let $\psi(t_1,\cdots,t_r) := \min_i \lvert t_i - (n+\tfrac{1}{2}) \rvert$,
let $\sigma(t_1,\cdots,t_r) := {\sharp}\{\, i \mid \lvert t_i - (n+\tfrac{1}{2}) \rvert = \psi(t_1,\cdots,t_r) \,\}$,
and let $\chi(t_1,\cdots,t_r) := \max \{\, i \mid t_i < n \text{ and } \lvert t_i - \tfrac{n+1}{2} \rvert \text{ is minimal} \,\}$.
\begin{equation*}
\begin{aligned}
    &\psi(t_1,\cdots,t_r)=\min\limits_{i=1}^r\{|t_i-(n+\frac{1}{2})|\},\\
     &\sigma(t_1,\cdots,t_r)=\sharp\{t\in \{ t_1,\cdots,t_r\}\colon |t-(n+\frac{1}{2})|=\psi(t_1,\cdots,t_r)\},\\
     &\chi(t_1,\cdots,t_r)=\max\{i\in \{1,\cdots,r\}:t_i<n,\psi(t_1\cdots,t_r)<n-t_i\}.
\end{aligned}
\end{equation*}
It is clear that $\psi(t_1,\cdots,t_r)\in\{\frac{1}{2},\frac{3}{2},\cdots,\frac{2n-1}{2}\}$ is a $\frac{\Zbb}{2}$-valued function, $\sigma(t_1,\cdots,t_r)\in \{1,\cdots,r\}$ is a $\Zbb$-valued function, and $\chi(t_1,\cdots,t_r)$ takes values in $\{1,\cdots,r\}$.

Consider sequences $I=(t_1,\cdots,t_r)$ such that $\psi(I)=b+\frac{1}{2}$, $\sigma(I)=a$, and $\chi(I)=N$. These terms can be partitioned into disjoint subsets according to their first $N$ entries and their last $r-(N+a)$ entries. In particular, one obtains $a$ sequences by fixing identical first $N$ entries and last $r-(N+a)$ entries; see Table \ref{tab:psib-sigmaa-sequence}.
\begin{table}[htbp]
\centering
\begin{tabular}{cccc}
\hline
$\substack{(t_1,\cdots,t_r)}$&first $N$-terms    & middle $a$-terms&last $r-(N+a)$-terms\\
\hline
&$t_1,\cdots,t_{N},$ & $n-b,\cdots, n-b,n-b,$    & $t_{N+a+1},\cdots,t_r$ \\
&$t_1,\cdots,t_{N},$ & $n-b,\cdots,n-b,n+b+1,$    & $t_{N+a+1},\cdots,t_r$ \\
&$\cdots$             & $\cdots$                    & $\cdots$ \\
&$t_1,\cdots,t_{N},$ & $n+b+1,\cdots,n+b+1,n+b+1,$ & $t_{N+a+1},\cdots,t_r$ \\
\hline
\end{tabular}
\caption{One collection of sequences $I$ with $\psi(I)=b+\frac{1}{2}$, $\sigma(I)=a$ and $\chi(I)=N$}\label{tab:psib-sigmaa-sequence}
\end{table}

They satisfy the properties $t_{N+a+1}>\max\{n-b,n+b+1\}$ and $t_{N}<\min\{n-b,n+b+1\}$. From the definition of path map concatenation, one has $\rho_i=\alpha_i$ for $i\leq n$ and $\rho_i=\alpha_{2n+1-i}^{-1}$ for $i\geq n+1$. Hence, by replacing indices $t_i>n$, that is, the entries $t_i$ behind $n+b+1$ in Table \ref{tab:psib-sigmaa-sequence}, with $(2n+1-t_i)$ in the sequence, the next equation follows, where $k=\varphi_{n+b-1}(I)$ denotes the number of occurrences of $(n+b-1)$ in the sequence.
\begin{equation*}
\begin{aligned}
     &\frac{\prod\limits_{i=1}^N\langle \omega_i,\rho_{t_i}\rangle\prod\limits_{i=N+1}^{N+a-k}\langle\omega_{i},\rho_{n-b}\rangle\prod\limits_{i=N+a-k+1}^{N+a}\langle\omega_{i},\rho_{n+b+1}\rangle\prod\limits_{i=N+a+1}^{r}\langle\omega_{i},\rho_{t_{i}}\rangle}{\tau(t_1,\cdots,t_{N},n-b,\cdots,n+b+1,t_{N+a+1},\cdots,t_r)}\\
    =&\frac{(-1)^{r-N+k}\prod\limits_{i=1}^N\langle \omega_i,\rho_{t_i}\rangle}{\tau(t_1,\cdots,t_{N})}\frac{\prod\limits_{i=N+1}^{N+a}\langle\omega_{i},\alpha_{n-b}\rangle}{k!(a-k)!}\frac{\prod\limits_{i=N+a+1}^{r}\langle\omega_{i},\alpha_{2n+1-t_{i}}\rangle}{\tau(t_{N+a+1},\cdots,t_r)}.
    \end{aligned}
\end{equation*}
Hence, the sum of $a$ terms in Table \ref{tab:psib-sigmaa-sequence} is the following
\begin{equation*}
    \begin{aligned}        0=\sum\limits_{k=0}^a&\frac{\prod\limits_{i=1}^N\langle \omega_i,\rho_{t_i}\rangle\prod\limits_{i=N+1}^{N+a-k}\langle\omega_{i},\rho_{n-b}\rangle\prod\limits_{i=N+a-k+1}^{N+a}\langle\omega_{i},\rho_{n+b+1}\rangle\prod\limits_{i=N+a+1}^{r}\langle\omega_{i},\rho_{t_{i}}\rangle}{\tau(t_1,\cdots,t_{N},n-b,\cdots,n+b+1,t_{N+a+1},\cdots,t_r)}\\
        =\sum\limits_{k=0}^a&\frac{(-1)^{r-N+k}}{k!(a-k)!}\frac{\prod\limits_{i=1}^N\langle \omega_i,\rho_{t_i}\rangle}{\tau(t_1,\cdots,t_{N})}\frac{\prod\limits_{i=N+1}^{N+a}\langle\omega_{i},\alpha_{n-b}\rangle}{1}\frac{\prod\limits_{i=N+a+1}^{r}\langle\omega_{i},\alpha_{2n+1-t_{i}}\rangle}{\tau(t_{N+a+1},\cdots,t_r)}.
    \end{aligned}
\end{equation*}
This came from the fact that the alternating sum of combinations vanishes, i.e.
\begin{equation*}
    \sum_{k=0}^a(-1)^k\binom{a}{k}=\sum_{k=0}^a\frac{a!(-1)^k}{k!(a-k)!}=0.
\end{equation*}
Moreover, the sum over all terms with $\psi(t_1,\cdots,t_r)=b+\frac{1}{2}$ and $\sigma(t_1,\cdots,t_r)=a$ also equals zero. Hence,
the iterated path integral vanishes, i.e.,
\begin{equation*}\label{eq:omega}
	\begin{aligned}
	&\ \int_{\rho}\omega_1\cdots\omega_r=\sum_{1\leq t_1\leq\cdots\leq t_r\leq 2n}\frac{\prod\limits_{i=1}^r\langle\omega_i,\rho_{t_i}\rangle}{\tau(I)}
	=\sum\limits_{b=0}^{n-1}\sum\limits_{a=1}^n\sum\limits_{\substack{I=(t_1,\cdots,t_r)\\\psi(I)=b+\frac{1}{2}\\\sigma(I)=a}}\frac{\prod\limits_{i=1}^r\langle\omega_i,\rho_{t_i}\rangle}{\tau(I)}=0 .
	\end{aligned}
\end{equation*}
\end{proof}

\begin{lemma}\label{lem:redd-path-nonzero}
Let $\alpha\colon I_n\to G$ be a reduced path map. If
\begin{equation*}
    \int_{\alpha}\omega_1\cdots\omega_r=0
\end{equation*}
for any 1-forms $\omega_1,\cdots,\omega_r$, then $\alpha$ is a trivial path map.
\end{lemma}

\begin{proof}
Assume that there exists a reduced path map $\alpha$ such that
\[
\int_{\alpha} \omega_1 \cdots \omega_r = 0
\]
for all $1$-forms $\omega_1, \ldots, \omega_r$ on $G$.
It follows that $\alpha_i \neq \alpha_{i+1}$ for all $i$.
Since $\alpha$ is reduced, the concatenation
$\alpha_i \star \alpha_{i+1}$ is not elementarily equivalent to a trivial path map for any $0 \le i \le n-1$.

For any arrow $\alpha\in G$, let $e^{\alpha_i}$ be the $1$-form defined by
\[
e^{\alpha}(\alpha)=1,\qquad
e^{\alpha}(\alpha^{-1})=-1,
\]
and $e^{\beta}=0$ for all arrow $\alpha \neq \beta  \in G $.
Consider the iterated path integral
\[
\int_{\alpha} e^{\alpha_1} e^{\alpha_2} \cdots e^{\alpha_n}
=
\sum_{1 \le t_1 \le \cdots \le t_n \le n}
\frac{\prod_{i=1}^n \langle e^{\alpha_i}, \alpha_{t_i} \rangle}
{\tau(t_1,\dots,t_n)}.
\]

Only the term $(t_1,\dots,t_n)=(1,\dots,n)$ contributes nontrivially.
All other terms vanish since no two adjacent arrows on $\alpha$ are equal or reverse to each other. 

Hence,
\[
\int_{\alpha} e^{\alpha_1} e^{\alpha_2} \cdots e^{\alpha_n}
= \pm 1 \neq 0,
\]
a contradiction. The lemma follows.
\end{proof}

\vspace{1em}

\begin{theorem}\label{thm:paths-equ-int}
Let $\alpha,\beta \in {P_{x,y}(G)}$. Two path maps $\alpha$ and $\beta$ are elementarily equivalent if and only if
\[
\int_{\alpha}\omega_1\cdots\omega_r=\int_{\beta}\omega_1\cdots\omega_r
\]
for any $1$-forms $\omega_1,\cdots,\omega_r$, $r\geq 1$.
\end{theorem}

\begin{proof}
Corollary~\ref{trivial} ensures invariance under insertion of trivial arrows, and Proposition~\ref{prop:iter-ele-equ} ensures invariance under elementary equivalence.

Conversely, Lemma~\ref{lem:redd-path-nonzero} shows that a reduced path map is elementarily equivalent to a trivial path map iff all iterated path integrals vanish.

Assume there exist distinct reduced path maps $\alpha$ and $\beta$ such that
\[
\int_{\alpha} \omega_1 \cdots \omega_r
=
\int_{\beta} \omega_1 \cdots \omega_r
\]
for all $1$-forms, but $\alpha \nsim \beta$.
Taking $\omega = df$ gives
\[
f(\alpha(n)) - f(\alpha(0)) = f(\beta(m)) - f(\beta(0)),
\]
for all $f \in \Acal(G)$.
Hence
\[
\alpha(0)=\beta(0), \qquad \alpha(n)=\beta(m).
\]

By Proposition~\ref{prop:ii-conca-path},
\[
\int_{\alpha \star \beta^{-1}} \omega_1 \cdots \omega_r
=
\int_{\alpha \star \alpha^{-1}} \omega_1 \cdots \omega_r
=0.
\]

Thus $\alpha \star \beta^{-1}$ is elementarily equivalent to a trivial path map, implying $\alpha \sim \beta$, a contradiction.
\end{proof}

\begin{corollary}
If
\begin{equation*}
    \int_{\alpha}\omega_1\cdots\omega_r=0
\end{equation*}
for all $1$-forms $\omega_1,\cdots,\omega_r$, then $\alpha$ is elementarily equivalent to the trivial path map.
\end{corollary}
\subsection{Bilinear Form and Adjointness}
\begin{definition}
Let $x,y$ be fixed vertices of $G$. The \emph{path iterated algebras} $\Pcal(G)$, $\Pcal_x(G)$ and $\Pcal_{x,y}(G)$ are the $\Kbb$-algebras generated by the indefinite iterated path integrals $\int \omega_1 \cdots \omega_r$, $\int_x \omega_1 \cdots \omega_r$ and $\int_x^y \omega_1 \cdots \omega_r$, respectively, $r \ge 1$, with multiplication given by the shuffle product
\begin{equation*}
  \int_x \omega_1 \cdots \omega_r \,
\int_x \omega_{r+1} \cdots \omega_{r+s}
=
\sum_{\sigma \in \Sh(r,s)}
\int_x \omega_{\sigma(1)} \cdots \omega_{\sigma(r+s)},
\end{equation*}
where $\Sh(r,s)$ denotes the set of $(r,s)$-shuffles. The unit is the constant function $1$.

The \emph{loop iterated algebras} $\Qcal(G)$ and $\Qcal_x(G)$ are the $\Kbb$-algebras generated by the indefinite loop iterated path integrals $\oint \omega_1 \cdots \omega_r$ and $\oint_x \omega_1 \cdots \omega_r$, respectively, for $r \ge 1$, with unit the constant function $1$.
\end{definition}

Let $\bar{P}(G)$ (resp. $\bar{P}_x(G)$, $\bar{P}_{x,y}(G)$, $\bar{L}(G)$, $\bar{L}_x(G)$) denote the sets of elementarily equivalent classes of $P(G)$ (resp. $P_x(G)$, $P_{x,y}(G)$, $L(G)$, $L_x(G)$). It follows that the dual maps (\ref{eq:dual-PG}) induce the following bilinear pairings
\begin{equation}\label{eq:bilinear-form}
	\begin{aligned}
    		\Pcal(G)\times \bar{P}(G)\to \Kbb,&\quad&\Pcal_x(G)\times \bar{P}_x(G)\to \Kbb,&\quad& \Pcal_{x,y}(G)\times \bar{P}_{x,y}(G)\to \Kbb,\\
            \Qcal(G)\times \bar{L}(G)\to \Kbb ,&\quad&\Qcal_{x}(G)\times \bar{L}_x(G)\to \Kbb.&\quad&
	\end{aligned}
\end{equation}
These pairings are denoted by $(-,-)$. For simplicity, we focus on $\Pcal_x(G)$, while the analogous statements hold for the other cases.

{By Lemma \ref{lem:redd-path-nonzero} and Proposition \ref{thm:paths-equ-int}, the kernel of the bilinear pairings in (\ref{eq:bilinear-form}) consists precisely of path maps that are elementarily equivalent to trivial path maps. Thus, the bilinear pairings in (\ref{eq:bilinear-form}) are well-defined and non-degenerate.}

Let {$G^*=(G_0,G_1,x)$, $H^*=(H_0,H_1,y)$ } be two based digraphs, and let $f\colon G^{*}\to H^{*}$ be a based digraph map. Let $\alpha\colon I_n\to G$ be a path map on $G$ with $\alpha(0)=x$, and let $\omega_1,\cdots,\omega_r$ be 1-forms on $H$. It follows that the composition $f\circ \alpha$ is elementarily trivial whenever $\alpha$ is elementarily trivial. Hence, $f$ induces a map
\begin{equation*}
	\begin{aligned}
	f_{*}\colon \bar{P}_x(G)\to \bar{P}_{f(x)}(H),\qquad
	[\alpha]\mapsto [f\circ \alpha].
	\end{aligned}
\end{equation*}

Let $f^{*}\colon \Omega^1(H)\to \Omega^1(G)$ be the pullback of 1-forms. The following proposition states that $f_{*}$ and $f^{*}$ are adjoint with respect to the pairing (\ref{eq:bilinear-form}).

\begin{proposition}\label{prop:dual-map}
Let $G^*=(G_0,G_1,x)$ and $H^*=(H_0,H_1,y)$ be based digraphs, and let $f\colon G\to H$ be a based digraph map. Let $\omega_1,\cdots,\omega_r$ be 1-forms on $H$. Then
\begin{equation}
	\int_{\alpha}f^{*}\omega_1\cdots f^{*}\omega_r=\int_{f_{*}\alpha}\omega_1\cdots\omega_r,
\end{equation}
equivalently,
$(f^{*}(\omega_1\cdots \omega_r),\alpha)=(\omega_1\cdots \omega_r,f_{*}\alpha)$,
where $f^{*}(\omega_1\cdots \omega_r):=f^{*}\omega_1\cdots f^{*}\omega_r$.
\end{proposition}

\begin{proof}
This follows from the identity
$f^{*}\omega_k(\alpha_t)=\omega_k(f_{*}\alpha_t)=\omega_k((f_{*}\alpha)_t)$.
\end{proof}

From the above proposition, the pullback $f^{*}$ extends to the iterated path algebra and iterated loop algebra as follows:
\begin{equation*}
	\begin{aligned}
f^{*}\colon \Pcal_{y}(H)&\to \Pcal_x(G),&\qquad&f^{*}\colon \Qcal_{y}(H)&\to \Qcal_{x}(G)\\
\int_{y}\omega_1\cdots\omega_r&\mapsto \int_x f^{*}(\omega_1\cdots\omega_r),&\qquad&\oint_{y}\omega_1\cdots\omega_r&\mapsto \oint_x f^{*}(\omega_1\cdots\omega_r).
	\end{aligned}
\end{equation*}

The following corollary is a direct consequence of Proposition \ref{prop:dual-map}.

\begin{corollary}\label{cor:covariant}
Let $G^*=(G_0,G_1,x)$, $H^*=(H_0,H_1,y)$ and $K^*=(K_0,K_1,z)$ be based digraphs, and let $f\colon G\to H$ and $g\colon H\to K$ be based digraph maps. Then
\begin{equation*}
  f^{\ast}g^{\ast}=(g\circ f)^{\ast}\colon \Pcal_z(K)\to \Pcal_x(G)\quad\text{resp. }  \Qcal_z(K)\to \Qcal_x(G).
\end{equation*}
\end{corollary}
\section{Algebras on Digraphs}
In this section, we introduce several algebras arising from the iterated path integrals of a digraph defined in previous section. In particular, we introduce $\pi^1(G)$, which is a homotopy invariant. It will be shown that these algebras carry Hopf algebra structures with antipodes.

To begin, we briefly recall some notions of Hopf algebras from \cite{dascalescu2000hopf}. A $\Kbb$-vector space $H$ is called a \textit{coalgebra} if it is equipped with two $\Kbb$-linear maps $\Delta_H\colon H\to H\otimes H$ and $\epsilon_H\colon H\to \Kbb$ satisfying:
\begin{enumerate}
	\item $\Delta_H$ is coassociative, i.e. $(\Delta_H\otimes 1_H)\Delta_H=(1_H\otimes\Delta_H)\Delta_H$,
	\item $\Delta_H$ is counitary, i.e. $(\epsilon_H\otimes 1_H)\Delta_H=(\epsilon_H\otimes\epsilon_H){\Delta_H} =1_H$.
\end{enumerate}
A $\Kbb$-linear map
$f\colon (H,\Delta_H,\epsilon_H)\to (H',\Delta_{H'},\epsilon_{H'})$
is a morphism of coalgebras if it satisfies $\Delta_{H'} f=(f\otimes f)\Delta_H,$ $\epsilon_{H'}f=\epsilon_H.$

A $\Kbb$-vector space $H$ is called a bialgebra if it is equipped with an algebra structure $(H,\mu_H\colon H\otimes H\to H,\eta_H\colon \Kbb\to H)$ and a coalgebra structure $(H,\Delta_H,\epsilon_H)$ such that $\Delta_H$ and $\epsilon_H$ are morphisms of algebras. It is well known in \cite[Chapter 4]{dascalescu2000hopf} that $\mu_H$ and $\eta_H$ are also morphisms of coalgebras. The map $\mu_H$ is called the multiplication of $H$, while $\Delta_H$ is called the comultiplication. The unit $\eta_H\colon \Kbb\to H$ and the counit $\epsilon_H$ are referred to as the unit and counit, respectively.

An antipode (or coinverse) of a bialgebra $H$ is a linear map $j_H\colon H\to H$ satisfying
\begin{equation*}
\mu_H(1_H\otimes j_H)\Delta_H=\eta_H\epsilon_H=\mu_H(j_H\otimes 1_H)\Delta_H.
\end{equation*}

A bialgebra admitting an antipode is called a Hopf algebra. If the antipode satisfies $j_H^2=\Id_H$, then the Hopf algebra is called involutive. Let $(H,\mu_H,\eta_H,\Delta_H,\epsilon_H,j_H)$ and $(H',\mu_{H'},\eta_{H'},\Delta_{H'},\epsilon_{H'},j_{H'})$ be Hopf algebras. A $\Kbb$-linear map $f\colon H\to H'$ is a morphism of Hopf algebras if it is both an algebra morphism and a coalgebra morphism, i.e., it preserves $(\mu,\eta)$ and $(\Delta,\epsilon)$. A standard result asserts that any bialgebra morphism satisfies
\[
j_{H'}\circ f=f\circ j_H,
\]
hence it preserves antipodes.

Let $H$ be a Hopf algebra. A subspace $H'\subset H$ is called a \emph{Hopf subalgebra} if it is simultaneously a subalgebra, a subcoalgebra, and satisfies $j_H(H')\subset H'$. In this case, $H'$ inherits a Hopf algebra structure from $H$.

\subsection{{Shuffle and Iterated Path Integral Algebras}}
In this subsection, we show that the $\Kbb$-algebra $\Qcal_x(G)$ admits a Hopf algebra structure with antipode. To verify this, let us begin with the shuffle algebra $\Sh(G)$.

Let $T(G)= \bigoplus_{n\geq 0} T^n(G)$ be the tensor algebra of $G$ over $\Kbb$, where $T^0(G)=\Omega^0(G)=\Acal(G)$ is the space of $\Kbb$-valued functions on $G$, and $T^1(G)=\Omega^1(G)$ is the space of $1$-forms on $G$. The $k$-th graded component $T^k(G)$ is a $\Kbb$-vector space with basis elements of the form
 $e^{a_1\cdots a_k}:=e^{a_1}\otimes\cdots \otimes e^{a_k},$
where $a_1,\cdots,a_k\in G_1$.

The graded concatenation product is defined by
\begin{equation}
	e^{a_1\cdots a_k}\cdot e^{b_1\cdots b_l}
	=
	e^{a_1\cdots a_k b_1\cdots b_l}.
\end{equation}

\begin{remark}
Note that the tensor product defined above is different from the exterior product of $p$-forms on a digraph introduced in \cite{grigor2015cohomology, grigor2018path}, which is given by
\begin{equation}
	e^{a_1\cdots a_k}\wedge e^{b_1\cdots b_l}
	=
	\begin{cases}
	0, & t(a_k)\neq s(b_1),\\[4pt]
	e^{a_1\cdots a_k b_1\cdots b_l}, & t(a_k)= s(b_1),
	\end{cases}
\end{equation}
where $a_1,\cdots,a_k,b_1,\cdots,b_l \in G_1$.
\end{remark}

\begin{definition}
The shuffle algebra $\Sh(G)$ of a digraph $G$ is the algebra whose underlying graded vector space is $T(G)$, and whose multiplication is defined by
\begin{equation}
e^{a_1\cdots a_r}\circ e^{a_{r+1}\cdots a_{r+s}}
=
\sum_{\sigma \in \Sh(r,s)}
e^{a_{\sigma(1)}\cdots a_{\sigma(r+s)}},
\end{equation}
where $\Sh(r,s)$ denotes the set of $(r,s)$-shuffles.
\end{definition}

Define $\epsilon_{\Sh(G)}\in \Hom(\Sh(G),\Kbb)$ by $\epsilon_{\Sh(G)}(1)=1$ and $\epsilon_{\Sh(G)}(\omega_1\cdots\omega_r)=0$ for $r\geq 1$. The comultiplication
\begin{equation*}
\Delta_{\Sh(G)}\colon \Sh(G)\to\Sh(G)\otimes\Sh(G)
\end{equation*}
is defined by $\Delta_{\Sh(G)}(1)=1$ and, for $r\geq 1$,
\begin{equation*}
\Delta_{\Sh(G)}(\omega_1\cdots\omega_r)
=
\sum_{i=0}^r (\omega_1\cdots\omega_i)\otimes(\omega_{i+1}\cdots\omega_r).
\end{equation*}
And the linear map $j_{\Sh(G)}$ is defined by
\begin{equation*}
j_{\Sh(G)}(\omega_1\cdots \omega_r)=(-1)^r\,\omega_r\cdots \omega_1.
\end{equation*} By \cite{dascalescu2000hopf,grinberg2014hopf}, the shuffle algebra $\Sh(G)$ is a Hopf $\Kbb$-algebra with counit $\epsilon_{\Sh(G)}$, comultiplication $\Delta_{\Sh(G)}$, and antipode $j_{\Sh(G)}$.

\begin{remark}
The tensor algebra $T(G)$ is also a Hopf algebra whose multiplication $\mu_{T(G)}$ and unit $\eta_{T(G)}$ are naturally defined. The comultiplication $\Delta_{T(G)}$ is given by $\Delta_{T(G)}(\omega_1\otimes\cdots\omega_m)=\sum_{p=1}^m\sum_{\sigma}(\omega_{\sigma (1)}\otimes\cdots\omega_{\sigma (p)}\boxtimes \omega_{\sigma (p+1)}\otimes\cdots\omega_{\sigma (m)}),$ where $\sigma$ {run} over all $(p,m-p)$-shuffles, $\otimes$ is the tensor product in $T(G)$ and $\boxtimes$ denote the tensor product between $T(G)$s. The counit $\epsilon_{T(G)}$ is given by $\epsilon_{T(G)}(1)=1$ and $\epsilon_{T(G)}(x)=0$ for $x \notin T^0(G)=\Kbb$. The shuffle algebra is the dual Hopf algebra of tensor algebra.
\end{remark}
Recall that the group algebra $\Kbb H$ is a Hopf algebra for any group $H$. In particular, $\Kbb \Bar{L}_x(G)=\Bar{L}_x(G)$  is a Hopf algebra for any digraph $G$, where the algebra structure
$\mu_{\Bar{L}_x(G)}\colon \Bar{L}_x(G)\otimes \Bar{L}_x(G)\to \Bar{L}_x(G)$ is given by loop concatenation, and the unit
$\eta_{\Bar{L}_x(G)}\colon \Kbb\to \Bar{L}_x(G)$ assigns $c\in\Kbb$ to $c\,e_x$, the constant loop at $x$. The coalgebra structure on $\Bar{L}_x(G)$ is given by
\[
\Delta_{\Bar{L}_x(G)}(g)=g\otimes g,\qquad
\epsilon_{\Bar{L}_x(G)}(g)=1_{\Kbb},\qquad
j_{\Bar{L}_x(G)}(g)=g^{-1}.
\]  

Fix a vertex $x\in G_0$. $\Pcal_x(G)$ (resp. $\Qcal_x(G)$) is the algebra generated by indefinite iterated path integrals $\int_x \omega_1\cdots\omega_r$ (resp. $\oint_x \omega_1\cdots\omega_r$), $r\geq 1$, together with the constant function 1 with the shuffle multiplication. There are surjective maps $\Sh(G)\to \Pcal_x(G)$ and $\Sh(G)\to \Qcal_x(G)$ given by, respectively,
\begin{equation}\label{eq:shP}
\begin{aligned}
\int_x\colon \Sh(G)\to\Pcal_x(G),&\quad \omega_1\cdots\omega_r \mapsto \int_x \omega_1\cdots\omega_r,\\
\oint_x\colon \Sh(G)\to \Qcal_x(G),&\quad \omega_1\cdots\omega_r \mapsto \oint_x \omega_1\cdots\omega_r.
\end{aligned}
\end{equation}
Define the homomorphisms of $\Kbb$-modules $$\epsilon_{\Qcal_x(G)}\colon \Qcal_x(G)\to \Kbb,\quad \oint_x u\mapsto \oint_{r_x}u,$$
where $r_x$ is the trivial arrow at $x$, $$\Delta_{\Qcal_x(G)}\colon \Qcal_x(G)\to \Qcal_x(G)\otimes \Qcal_x(G),\quad \oint_{x}\omega_1\cdots\omega_r\mapsto\sum_{0\leq i\leq r}\oint_x\omega_1\cdots\omega_i\otimes\oint_x\omega_{i+1}\cdots\omega_r$$ together with $$j_{\Qcal_x(G)}\colon \Qcal_x(G)\to \Qcal_x(G),\quad \oint_{x}\omega_1\cdots\omega_r\mapsto (-1)^r\oint_{x}\omega_r\cdots\omega_1.$$

\begin{proposition}\label{prop:Q-hopf}Under above homomorphisms,
$\Qcal_x(G)$ is a Hopf algebra with counit $\epsilon_{\Qcal_x(G)}$, comultiplication $\Delta_{\Qcal_x(G)}$, and antipode $j_{\Qcal_x(G)}$.
\end{proposition}

To prove Proposition~\ref{prop:Q-hopf}, we first recall the following lemma.
\begin{lemma}[Lemma 3.5, \cite{chen1971algebras}]
Let $H'$ be a Hopf algebra, and $(H, \mu_H, \eta_H)$ be an algebra equipped with linear maps
\[
\epsilon_H \colon  H \to \Kbb, \quad
\Delta_H \colon  H \to H \otimes H, \quad
j_H \colon  H \to H.
\]
Suppose there exists a left non-degenerate pairing
$\langle -, - \rangle \colon  H \times H' \to \Kbb
,$
which induces a left non-degenerate pairing
\[
H \otimes H \times H' \otimes H' \to \Kbb, \quad
\langle h_1 \otimes h_2, h'_1 \otimes h'_2 \rangle \mapsto \langle h_1, h'_1 \rangle \, \langle h_2, h'_2 \rangle,
\]
such that for all $h, h_1, h_2 \in H$, $h'_1, h'_2, h' \in H'$, and $c \in \Kbb$,
\begin{align*}
\langle \mu_H(h_1, h_2), h' \rangle = \langle h_1 \otimes h_2, \Delta_{H'} h' \rangle,\quad
\langle h, \mu_{H'}(h'_1, h'_2) \rangle = \langle \Delta_H h, h'_1 \otimes h'_2 \rangle, \\
\langle \eta_H c, h' \rangle = c \, \epsilon_{H'}(h'),\quad
\langle h, \eta_{H'} c \rangle = c \, \epsilon_H(h), \quad  \langle j_H h, h' \rangle = \langle h, j_{H'} h' \rangle.
\end{align*}
Then $H$ is a Hopf algebra, and $j_H$ is an antipode of $H$.
\end{lemma}

\begin{proof}[Proof of Proposition \ref{prop:Q-hopf}] As we know, the bilinear form $(-,-)\colon \Qcal_x(G)\times \Bar{L}_x(G)\to \Kbb$ is {non-degenerate}. It suffices to check that the morphisms $\epsilon_{\Qcal_x(G)}$, $\Delta_{\Qcal_x(G)}$, and $j_{\Qcal_x(G)}$ satisfy the required compatibility conditions with respect to the pairing
\begin{equation*}
    \Qcal_x(G)\times \Bar{L}(G)\to \Kbb .
\end{equation*}
These identities follow directly from the definitions, Proposition~\ref{prop:ii-conca-path}, and Proposition~\ref{prop:ii-inv-path}.
\end{proof}

\begin{theorem}\label{thm:PQ-func}
There exist two well-defined {contra}variant functors
\begin{equation*}
    \begin{aligned}
        \Pcal &\colon  \underline{\BDig} \to \underline{\mathcal{A}}, &\quad G^*=(G_0,G_1,x) &\mapsto \Pcal_x(G),\\
        \Qcal &\colon  \underline{\BDig} \to \underline{\mathcal{H}}, &\quad G^*=(G_0,G_1,x) &\mapsto \Qcal_x(G),
    \end{aligned}
\end{equation*}
where $\underline{\BDig}$ denotes the category of based digraphs, and $\underline{\mathcal{A}}$ and $\underline{\mathcal{H}}$ denote the categories of $\mathbb{K}$-algebras and Hopf $\mathbb{K}$-algebras, respectively.
\end{theorem}
\begin{proof}
    The functoriality comes from the Proposition \ref{prop:dual-map} and Corollary \ref{cor:covariant}. The $\Qcal_x(G)$ is a Hopf algebra came from the last Proposition.
\end{proof}

\subsection{Filtration}
Since the shuffle multiplication is compatible with the grading on $T(G)$, the shuffle algebra $\Sh(G)$ admits an ascending filtration
\[
F^r\Sh(G):=T^{0}(G)\oplus \cdots\oplus T^r(G), \quad r\geq 0.
\]
Denote by $F^r\Pcal_x(G)$ and $F^r\Qcal_x(G)$ the images of $F^r\Sh(G)$ under the canonical morphisms. Then the algebras $\Pcal_x(G)$ and $\Qcal_x(G)$ are filtered. Note that $F^r\Pcal_x(G)$ and $F^r\Qcal_x(G)$ consist of linear combinations of iterated path integrals of length at most $r$.

\begin{definition}
A path map $\alpha$ is said to be at least of order $r$ if $(\omega,\alpha)=0$ for any {non-constant element} $\omega\in F^{r-1} \Sh (G)$, $r\geq 1$. Here $(\omega,\alpha)$ is the bilinear pairing induced from (\ref{eq:bilinear-form}).
\end{definition}

Every path map is at least of order $1$. If a path map $\alpha$ is at
least of order $2$, then $\alpha$ is a loop. {Indeed, for every function
$f$, we have
\[
\int_{\alpha} df=0.
\]
Since
\[
\int_{\alpha} df=f(\alpha(n))-f(\alpha(0)),
\]
it follows that $\alpha(n)=\alpha(0)$. Conversely, the order of a loop may not be $2$. Moreover, path maps of higher
order also exist. An example of a path map of order at least $3$ is given below.
}
{
\begin{example}
Let $G$ be the wedge of three directed cycles illustrated below. There are three loops $
a=xu_1u_2x,$ $b=xv_1v_2x,$ and $ c=xr_1r_2x$.
Then it follows that there is a loop $\gamma = [c, [a, b]] = caba^{-1}b^{-1}c^{-1}bab^{-1}a^{-1}.$ We claim that $\gamma$ is at least order $3$.
$$
\begin{tikzpicture}[scale=0.7, >=stealth, thick, every node/.style={font=\small}]
    \node (x)  at (0,0)   {$x$};
    \node (u1) at (1.4,1.2) {$u_1$};
    \node (u2) at (1.4,-1.2) {$u_2$};
    \node (v1) at (-0.4,-1.8) {$v_1$};
    \node (v2) at (-1.4,-0.8) {$v_2$};
    \node (r1) at (-2.1,1.4) {$r_1$};
    \node (r2) at (-0.4,1.8) {$r_2$};
    \draw[->] (x) -- (u1);
    \draw[->] (u1) -- (u2);
    \draw[->] (u2) -- (x);
    \draw[->] (x) -- (v1);
    \draw[->] (v1) -- (v2);
    \draw[->] (v2) -- (x);
    \draw[->] (x) -- (r1);
    \draw[->] (r1) -- (r2);
    \draw[->] (r2) -- (x);
\end{tikzpicture}
$$
It is easy to check that for any $1$-form $\omega$, $(\omega,[a,b]) = 0$, which implies $\beta = [a,b]$ is at least order $2$. Next we check that for any two $1$-forms $\omega_1, \omega_2\in \Omega^1(G),$ $(\omega_1 \omega_2, \gamma) = 0. $ 
By Proposition~3.5,
\begin{align*}
\int_{c \beta c^{-1}}\omega_1\omega_2
=
\int_{c\star\beta}\omega_1\omega_2
+\left(\int_{c\star\beta}\omega_1\right)\left(\int_{c^{-1}}\omega_2\right)
+\int_{c^{-1}}\omega_1\omega_2\\
=
\int_c\omega_1\omega_2+\int_\beta\omega_1\omega_2
-\left(\int_c\omega_1\right)\left(\int_c\omega_2\right)
+\int_c\omega_2\omega_1 = \int_\beta\omega_1\omega_2.
\end{align*}
Further,
\begin{align*}
\int_\gamma\omega_1\omega_2
&=
\int_{c \beta c^{-1}}\omega_1\omega_2
+\left(\int_{c \beta c^{-1}}\omega_1\right)\left(\int_{\beta^{-1}}\omega_2\right)
+\int_{\beta^{-1}}\omega_1\omega_2\\
&=
\int_\beta\omega_1\omega_2+\int_{\beta^{-1}}\omega_1\omega_2=\int_\beta\omega_1\omega_2+\int_\beta\omega_2\omega_1.
\end{align*}
Applying the shuffle identity again, we get
\[
\int_\beta\omega_1\omega_2+\int_\beta\omega_2\omega_1
=
\left(\int_\beta\omega_1\right)\left(\int_\beta\omega_2\right)=0.
\]
Thus,
$\gamma$ is of at least of order $3$.
\end{example}
}
Denote by $F_rL_x(G)$ the set of loops at $x$ that are at least of order $r$, and by $F_r\Bar{L}(G,x)$ the set of elementary equivalence classes of such loops. The following lemma follows directly from the definition.

\begin{lemma}
A path map $\alpha$ is at least of order $r$ if and only if $(u,\alpha)=0$ for any $u\in F^{r-1}\Sh(G)\cap \ker\epsilon_{\Sh(G)}$.
\end{lemma}

From Proposition \ref{prop:ii-inv-path}, if $\alpha\in F_rL_x(G)$ for some $r\geq 1$, then its inverse $\alpha^{-1}$ is also at least of order $r$.

\begin{remark}
Clearly, $\Bar{L}_x(G)$ admits a descending filtration
\[
\Bar{L}_x(G)=F_1\Bar{L}_x(G) \supseteq F_2\Bar{L}_x(G)\supseteq \cdots.
\]
\end{remark}

\begin{corollary}\label{cor:filtered}
Let $\alpha\in F_r\Bar{L}_x(G)$ and let $\beta$ be a path map from $x$ that is at least of order $s$. If $u\in F^{r+s-1}\Sh(G)\cap \ker\epsilon_{\Sh(G)}$, then
\[
(u,\alpha\star\beta)=(u,\alpha)+(u,\beta).
\]
\end{corollary}

\begin{proof}
Let $u=\omega_1\cdots\omega_l$ with $l\leq r+s-1$. By Proposition \ref{prop:ii-conca-path}, we have
\[
\int_{\alpha\star\beta}u=\int_{\alpha}u+\int_{\beta}u.
\]
\end{proof}

\begin{corollary}\label{cor:iter-invs}
If $\alpha\in F_r\Bar{L}_x(G)$ and $u\in F^{2r-1}\Sh(G)\cap \ker\epsilon_{\Sh(G)}$, then
\[
(u,\alpha^{-1})=-(u,\alpha).
\]
In particular, if $u\in F^{r-1}\Sh(G)\cap \ker\epsilon_{\Sh(G)}$, then
\[
(u,\alpha^{-1})+(u,\alpha)=0.
\]
\end{corollary}

\begin{proof}
This follows from Corollary \ref{cor:filtered} by taking $\beta=\alpha^{-1}$.
\end{proof}

Let $\alpha$ and $\beta$ be loops on $G$, and write $[\alpha,\beta]:=\alpha\star\beta\star\alpha^{-1}\star\beta^{-1}$.

\begin{lemma}\label{lem:iter-liebre-path}
If $\alpha\in F_rL_x(G)$ and $\beta\in F_sL_x(G)$ with $r\leq s$, then $[\alpha,\beta]\in F_{r+s}L_x(G)$ and
\[
(\omega_1\cdots\omega_{r+s},[\alpha,\beta])=(\omega_1\cdots\omega_r,\alpha)(\omega_{r+1}\cdots\omega_{r+s},\beta)-(\omega_1\cdots\omega_s,\beta)(\omega_{s+1}\cdots\omega_{r+s},\alpha).
\]
\end{lemma}

\begin{proof}
Let $m\geq 1$. Since $\alpha\alpha^{-1}$ and $\beta\beta^{-1}$ are elementarily equivalent to the trivial path map, for any $m$,
\begin{equation}\label{eq:sum-a-ainv}
\begin{aligned}
\sum_{i=0}^m(\omega_1\cdots\omega_i,\alpha)(\omega_{i+1}\cdots\omega_m,\alpha^{-1})=0,~
\sum_{i=0}^m(\omega_1\cdots\omega_i,\beta)(\omega_{i+1}\cdots\omega_m,\beta^{-1})=0.
\end{aligned}
\end{equation}

From the definition of the filtration, $(\omega_1\cdots\omega_m,\alpha)=0$ for $1\leq m\leq r-1$ and $(\omega_1\cdots\omega_m,\beta)=0$ for $1\leq m\leq s-1$. Hence
\begin{equation}\label{eq:sum-a-binv}
\begin{aligned}
\sum_{i=0}^m(\omega_1\cdots\omega_i,\alpha)(\omega_{i+1}\cdots\omega_m,\beta^{-1})=0,\quad m\leq r+s-1,\\
\sum_{i=0}^m(\omega_1\cdots\omega_i,\beta)(\omega_{i+1}\cdots\omega_m,\alpha^{-1})=0,\quad m\leq r+s-1.
\end{aligned}
\end{equation}

Assume $m\leq r+s-1$. By Proposition \ref{prop:ii-conca-path} and Corollary \ref{cor:filtered},
\begin{equation*}
\begin{aligned}
(\omega_1\cdots\omega_m,[\alpha,\beta])
=\sum_{i=0}^{m}(\omega_1\cdots\omega_i,\alpha\star\beta)(\omega_{i+1}\cdots\omega_m,\alpha^{-1}\star\beta^{-1})\\
=\sum_{i=0}^{m}[(\omega_1\cdots\omega_i,\alpha)(\omega_{i+1}\cdots\omega_m,\alpha^{-1})
+(\omega_1\cdots\omega_i,\beta)(\omega_{i+1}\cdots\omega_m,\alpha^{-1})\\
+(\omega_1\cdots\omega_i,\alpha)(\omega_{i+1}\cdots\omega_m,\beta^{-1})
+(\omega_1\cdots\omega_i,\beta)(\omega_{i+1}\cdots\omega_m,\beta^{-1})].
\end{aligned}
\end{equation*}
Using (\ref{eq:sum-a-ainv}) and (\ref{eq:sum-a-binv}), this vanishes. Hence $[\alpha,\beta]\in F_{r+s}L_x(G)$.

For $m=r+s$, we have
\begin{equation}\label{eq:rpluss}
\begin{aligned}
\ &(\omega_1\cdots\omega_{r+s},\alpha\star\beta)
=(\omega_1\cdots\omega_{r+s},\alpha)
+(\omega_1\cdots\omega_{r+s},\beta)\\
&+(\omega_1\cdots\omega_r,\alpha)(\omega_{r+1}\cdots\omega_{r+s},\beta),\\
\ &(\omega_1\cdots\omega_{r+s},\alpha^{-1}\star\beta^{-1})
=(\omega_1\cdots\omega_{r+s},\alpha^{-1})
+(\omega_1\cdots\omega_{r+s},\beta^{-1})\\
&+(\omega_1\cdots\omega_r,\alpha^{-1})(\omega_{r+1}\cdots\omega_{r+s},\beta^{-1}).
\end{aligned}
\end{equation}By Corollary~\ref{cor:iter-invs},
\[
(\omega_1\cdots\omega_r,\alpha)
+(\omega_1\cdots\omega_r,\alpha^{-1})=0,\quad
(\omega_{s+1}\cdots\omega_{r+s},\alpha^{-1})
+(\omega_{s+1}\cdots\omega_{r+s},\alpha)=0,
\]
which imply
\begin{equation}\label{eq:0-lots}
\begin{aligned}
0
=&(\omega_1\cdots\omega_r,\alpha^{-1})(\omega_{r+1}\cdots\omega_{r+s},\beta^{-1})
+(\omega_1\cdots\omega_r,\alpha)(\omega_{r+1}\cdots\omega_{r+s},\beta^{-1})\\
&+(\omega_1\cdots\omega_s,\beta)(\omega_{s+1}\cdots\omega_{r+s},\alpha^{-1})
+(\omega_1\cdots\omega_s,\beta)(\omega_{s+1}\cdots\omega_{r+s},\alpha).
\end{aligned}
\end{equation}
Then
\begin{equation*}
\begin{aligned}
(\omega_1\cdots\omega_{r+s},[\alpha,\beta])
=&(\omega_1\cdots\omega_{r+s},\alpha\star\beta)
+(\omega_1\cdots\omega_{r+s},\alpha^{-1}\star\beta^{-1})\\
&+\sum_{i=1}^{r+s-1}
(\omega_1\cdots\omega_i,\alpha\star\beta)
(\omega_{i+1}\cdots\omega_{r+s},\alpha^{-1}\star\beta^{-1}).
\end{aligned}
\end{equation*}
Applying (\ref{eq:rpluss}) and the filtration property, we obtain
\begin{equation*}
\begin{aligned}
(\omega_1\cdots\omega_r,\alpha)(\omega_{r+1}\cdots\omega_{r+s},\beta)
+(\omega_1\cdots\omega_r,\alpha^{-1})(\omega_{r+1}\cdots\omega_{r+s},\beta^{-1})\\
+(\omega_1\cdots\omega_r,\alpha)(\omega_{r+1}\cdots\omega_{r+s},\beta^{-1})
+(\omega_1\cdots\omega_s,\beta)(\omega_{s+1}\cdots\omega_{r+s},\alpha^{-1}).
\end{aligned}
\end{equation*}
Using (\ref{eq:0-lots}), this reduces to
\begin{equation*}
(\omega_1\cdots\omega_r,\alpha)(\omega_{r+1}\cdots\omega_{r+s},\beta)
-(\omega_1\cdots\omega_s,\beta)(\omega_{s+1}\cdots\omega_{r+s},\alpha).
\end{equation*}
\end{proof}
\begin{proposition}
The space $F_r\Bar{L}_x(G)$ is a normal subgroup of $\Bar{L}_x(G)$ and
\begin{equation*}
	[F_r\Bar{L}_x(G),F_s\Bar{L}_x(G)]\subset F_{r+s}\Bar{L}_x(G).
\end{equation*}
	\end{proposition}
\begin{proof}
Corollary \ref{cor:filtered} implies that $F_r\Bar{L}_x(G)$ is a subgroup of $\Bar{L}_x(G)$, and Lemma \ref{lem:iter-liebre-path} yields
\[
[F_r\Bar{L}_x(G),F_s\Bar{L}_x(G)]\subset F_{r+s}\Bar{L}_x(G).
\]
It remains to verify that $F_r \Bar{L}_x(G)$ is normal.
For any $\alpha\in F_r\Bar{L}_x(G)$, $\beta\in \Bar{L}_x(G)$, and $u\in F^{r-1}\Sh(G)\cap\ker\epsilon_{\Sh(G)}$,
$(u,\beta\star\alpha\star\beta^{-1})
=(u,[\beta,\alpha]\star\alpha)
=(u,[\beta,\alpha])+(u,\alpha)=0.$
\end{proof}
\section{Homotopy Invariant Hopf Algebras $\pi^1$}
Let $G=(G_0,G_1)$ be a connected digraph with fixed vertices $x,y$. Let $\omega_1,\cdots,\omega_r$ be 1-forms on $G$.
\begin{definition}
An element {$\int_x \omega_1\cdots\omega_r$} is called \emph{independent of $C_{\partial}$-homotopy},  if for any path maps $\alpha,\alpha'$ from $x$ to $y$ with $\alpha \simeq_{C} \alpha'$, {$\int_{\alpha}\omega_1\cdots\omega_r=\int_{\alpha'}\omega_1\cdots\omega_r$}.
\end{definition}
\begin{definition}
An element $\oint_x \omega_1\cdots\omega_r$ is called \emph{independent of $C_{\partial}$-homotopy}, if for any loops $\alpha,\alpha'$ based at $x$ with $\alpha \simeq_{C} \alpha'$ {$\oint_{\alpha} \omega_1\cdots\omega_r=\oint_{\alpha'}\omega_1\cdots\omega_r$}  .
\end{definition}
The elements $\int_x \omega_1\cdots\omega_r$ that are independent of $C_{\partial}$-homotopy form a subalgebra $\Gamma_x(G)$ of $\Pcal_x(G)$. Similarly, the elements $\oint_x \omega_1\cdots\omega_r$ with this property generate a subalgebra $\pi^1_x(G)$ of $\Qcal_x(G)$.

Let $\omega$ be a closed 1-form on $G$, i.e. $d\omega=0$ in $\Omega_2(G)$. The following result shows that $\pi^1_x(G)$ is non-empty; in particular, it contains all closed 1-forms.
\begin{lemma}\label{lem:tri-square}
Let $x$ be a fixed vertex and $\omega$ a 1-form on $G$. The element $\int_x \omega$ (resp. $\oint_x \omega$) is independent of $C_{\partial}$-homotopy if and only if $\omega$ satisfies:
\begin{enumerate}
    \item $\langle \omega, a_1 \star a_2 \rangle = \langle \omega, a_3 \rangle$ for any subdigraph isomorphic to a standard triangle,
    \item $\langle \omega, a_1 \star a_2 \rangle = \langle \omega, a_3 \star a_4 \rangle$ for any subdigraph isomorphic to a standard square,
    \item $\langle \omega, a_1 \star a_2 \rangle = 0$ for any subdigraph isomorphic to a standard double edge.
\end{enumerate}
\end{lemma}
\begin{proof}
The necessity follows from the relations $a_1\star a_2\simeq_C a_3$ in a triangle, $a_1\star a_2\simeq_C a_3\star a_4$ in a square, and $a_1\star a_2\simeq_C r_{v_0}$ in the double edge case.

For the converse, let $\alpha \colon I_n \to G$ and $\alpha' \colon I_m \to G$ be path maps from $x$ to $y$ with $m \ge n$, and assume they are one-step $C_{\partial}$-homotopic. Then there exist maps $h \colon I_m \to I_n$ and $H \colon C_h \to G$ or $H \colon C_h^- \to G$ such that $H|_{I_n} = \alpha$ and $H|_{I_m} = \alpha'$. Since $C_h$ and $C_h^-$ decompose into triangles and squares, the conditions above imply
\begin{equation*}
    \int_{\alpha} \omega = \int_{\alpha'} \omega .
\end{equation*}
\end{proof}

\begin{proposition}\label{prop:closed-ind-path}
Let $G^{*}=(G_0,G_1,x)$ be a based digraph and $\omega$ a 1-form on $G$. The element $\int_x \omega$ (resp. $\oint_x\omega$) is independent of $C_{\partial}$-homotopy if and only if $\omega$ is closed.
\end{proposition}

\begin{proof}
Assume that $\omega$ is closed. Then for any $\sigma\in \Omega_2(G)$,
$0=(d\omega,\sigma)=(\omega,\partial \sigma),$
which implies the conditions in Lemma \ref{lem:tri-square}.

Conversely, it is shown in \cite{grigor2020path,grigor2014homotopy} that $\Omega_2(G)$ is generated by triangles, squares, and double edges. Hence the above conditions imply $(\omega,\partial \sigma)=0$ for all $\sigma\in \Omega_2(G)$, and therefore $d\omega=0$.
\end{proof}

Let $\omega_1\cdots\omega_r$ be an element of $T^r(G)\subset\Sh(G)$. The 1-form $\omega_1\cdots\omega_r$ is called \textit{triangular isosceles} if for any digraph embedding $\iota\colon G'\hookrightarrow G$ such that $\Img \iota$ is a standard triangle in $G$, and any tuple $(b_1,\cdots,b_r)$ with $b_i\in\{a_1,a_2\}$, the following equality holds for any $\sigma\in S_r$.
\begin{equation}\label{eq:isoscele}
   \langle \omega_1\cdots\omega_r,(b_{\sigma (1)},\cdots,b_{\sigma (r)})\rangle:=
   \Pi_{i=1}^r\langle \omega_i,b_{\sigma (i)}\rangle
   =
   \Pi_{i=1}^r\langle \omega_i,b_{i}\rangle
   =
   \langle \omega_1\cdots\omega_r,(b_{1},\cdots,b_{r})\rangle
\end{equation}
The element $\omega_1\cdots\omega_r$ is called \textit{square isosceles} if for any digraph embedding $\iota\colon G'\hookrightarrow G$ such that $\Img\iota$ is a standard square in $G$, and any tuple $(b_1,\cdots,b_r)$ with either $b_i\in\{a_1,a_2\}$ or $b_i\in\{a_3,a_4\}$, Equation (\ref{eq:isoscele}) holds for all $\sigma\in S_r$. Moreover, the element $\omega_1 \cdots \omega_r$ is called \textit{isosceles} if it is both triangular isosceles and square isosceles.

\begin{lemma}\label{lem:tri-squre-ind}
Assume that $G$ is a triangle or a square, and let $\omega_1, \ldots, \omega_r$ be closed 1-forms on $G$ such that $\omega_1 \cdots \omega_r \in \Sh(G)$ is isosceles. Then $
    \int_x \omega_1 \cdots \omega_r ~(\text{resp. } \oint_x \omega_1 \cdots \omega_r)$
is independent of $C_{\partial}$-homotopy.
\end{lemma}
\begin{proof}
If $G$ is a triangle, then $\omega_i(a_3)=\omega_i(a_1)+\omega_i(a_2)$. Since $\omega_1\cdots\omega_r$ is isosceles,
\begin{equation*}
    \frac{\Pi_{i=1}^r\omega_i(a_3)}{r!}
    =
    \frac{1}{r!}\sum_{k=0}^{r}\binom{r}{k}\omega_1(a_1)\cdots\omega_k(a_1)\omega_{k+1}(a_2)\cdots\omega_r(a_2)
    =
    \int_{a_1\star a_2}\omega_1\cdots\omega_r.
\end{equation*}

If $G$ is a square, then $\omega_i(a_1)+\omega_i(a_2)=\omega_i(a_3)+\omega_i(a_4)$. A similar computation gives
\begin{equation*}
	\begin{aligned}
	\frac{\Pi_{i=1}^r(\omega_i(a_1)+\omega_i(a_2))}{r!}
	&=
	\frac{1}{r!}\sum_{k=0}^r\binom{r}{k}\omega_1(a_1)\cdots\omega_k(a_1)\omega_{k+1}(a_2)\cdots\omega_r(a_2)
	=
	\int_{a_1\star a_2}\omega_1\cdots\omega_r,\\
	\frac{\Pi_{i=1}^r(\omega_i(a_3)+\omega_i(a_4))}{r!}
	&=
	\frac{1}{r!}\sum_{k=0}^r\binom{r}{k}\omega_1(a_3)\cdots\omega_k(a_3)\omega_{k+1}(a_4)\cdots\omega_r(a_4)
	=
	\int_{a_3\star a_4}\omega_1\cdots\omega_r.
	\end{aligned}
\end{equation*}
\end{proof}
\begin{theorem}\label{thm:closed-inds-path}
Let $G^{*}=(G_0,G_1,x)$ be a based digraph and let $\omega_1,\cdots,\omega_r$ be $1$-forms on $G$ such that
\begin{enumerate}
    \item $\omega_i$ is closed for any $i$,
    \item the element $\omega_i\omega_{i+1}\cdots\omega_j$ is isosceles for any $1\leq i\leq j\leq r$.
\end{enumerate}
Then the element $\int_x \omega_1\cdots\omega_r$ (resp.\ $\oint_x \omega_1\cdots\omega_r$) is independent of $C_{\partial}$-homotopy.
\end{theorem}

\begin{proof}
It suffices to treat $\int_x \omega_1\cdots\omega_r$. Let $y$ be a vertex of $G$, and let $\alpha:I_n\to G$, $\alpha':I_m\to G$ ($n\geq m$) be path maps from $x$ to $y$.

By definition of $C_{\partial}$-homotopy, it is enough to consider the case where $\alpha$ and $\alpha'$ are one-step $C_{\partial}$-homotopic. Hence there exists a shrinking map $h:I_n\to I_m$ and a digraph map $F:C_h\to G$ such that $F|_{I_n}=\alpha$ and $F|_{I_m}=\alpha'$.

We construct a sequence of path maps $\psi_i$ on $C_h$. Set
$\psi_0=\alpha\star(\alpha(n)\to \alpha'(m)).$ Since $h(n)=m$ and $\alpha(n)=\alpha'(m)=y$, we have $F_*\psi_0=\alpha$.

Inductively, $\psi_{i}$ is obtained from $\psi_{i-1}$. If $[\alpha(n-i)\alpha(n-i+1)\alpha'(h(n-i))]$ or $[\alpha(n-i+1)\alpha(n-i)\alpha'(h(n-i))]$ forms a triangle, one example is shown as follows:
\begin{equation*}\label{eq:tri-mapping-cone}
    \xymatrix{
    \alpha(0)\ar[r]\ar[d]^h & \cdots\ar[r] &\shortstack{$\alpha(n-i)$\\=$\psi_{i-1}(n-i)$\\=$\psi_i(n-i)$ }\ar[r]\ar[d]^h & \shortstack{$\alpha(n-i+1)$\\=$\psi_{i-1}(n-i+1)$}\ar[r]\ar[dl]^h & \cdots \ar[r] & \alpha(n)\ar[dl]^h \\
    \alpha'(0)\ar[r] & \cdots\ar[r] & \shortstack{$\alpha'(h(n-i))$\\=$\psi_{i-1}(n-i+2)$\\=$\psi_i(n-i+1)$} \ar[r] & \cdots \ar[r] &\alpha'(m) &
    }
\end{equation*}
then $\psi_i$ is the concatenation of $\alpha|_{[0,n-i]}$, arrow $(\alpha(n-i)\to \alpha'(h(n-i)))$ and $\psi_{i-1}|_{[n-i+2:]}$.

If $[\alpha(n-i)\alpha(n-i+1)\alpha'(h(n-i))\alpha'(h(n-i+1))]$ or $[\alpha(n-i+1)\alpha(n-i)\alpha'(h(n-i+1))\alpha'(h(n-i))]$ is a square (see one example below),
\begin{equation}\label{eq:sqr-mapping-cone}
\xymatrix{
\alpha(0)\ar[r]\ar[d]^h
& \cdots \ar[r]
& \shortstack{$\alpha(n-i)$\\$=\psi_{i-1}(n-i)$\\$=\psi_i(n-i)$}
  \ar[r]\ar[d]^h
& \shortstack{$\alpha(n-i+1)$\\$=\psi_{i-1}(n-i+1)$}
  \ar[r]\ar[d]^h
& \cdots \ar[r]
& \alpha(n)\ar[d]^h \\
\alpha'(0)\ar[r]
& \cdots \ar[r]
& \shortstack{$\alpha'(h(n-i))$\\$=\psi_i(n-i+1)$}
  \ar[r]
& \shortstack{$\alpha'(h(n-i+1))$\\$=\psi_{i-1}(n-i+2)$\\$=\psi_i(n-i+2)$}
  \ar[r]
& \cdots \ar[r]
& \alpha'(m)
}
\end{equation}
then $\psi_{i}$ is the concatenation of $\alpha|_{[0,n-i]}$, the 2-step path maps $(\alpha(n-i)\to\alpha'(h(n-i))\to\alpha'(h(n-i+1)))$ and $\psi_{i-1}|_{[n-i+2:]}$.

{The path maps $\psi_i$ satisfy the following properties:
        \begin{enumerate}
            \item in the case $i=n$, $\psi_n$ is the concatenation of the arrow ${(\alpha(0)\to \alpha'(0))}$ and $\alpha'$
            \item for any $i$, $\psi_i$ and $\psi_{i+1}$ are 1-step $C_{\partial}$-homotopic,
            \item for any $i$, $\psi_i|_{[0,n-i]}=\alpha|_{[0,n-i]}=\psi_{i-1}|_{[0,n-i]}$ is the cut of $\alpha$ and $\psi_i|_{[n-i+2:]}=\psi_{i-1}|_{[n-i+2:]}$ is the cut of $\alpha'$
        \end{enumerate}
Then from Proposition \ref{thm:paths-equ-int} and Lemma \ref{lem:tri-squre-ind} there are $\int_{\psi_i}\omega_1\cdots\omega_r=\int_{\psi_{i+1}}\omega_1\cdots\omega_r$ for any $i=0,\cdots,n-1$.
Hence $\int_{\alpha}\omega_1\cdots\omega_r=\int_{\alpha'}\omega_1\cdots\omega_r$.}
\end{proof}

\subsection{Hopf Algebra Structure}
It is clear that the unit $\eta_{\Qcal_x(G)}$, the shuffle multiplication $\mu_{\Qcal_x(G)}$, the counit $\epsilon_{\Qcal_x(G)}$, and the antipode $j_{\Qcal_x(G)}$ restrict naturally to $\pi^1_x(G)$. Moreover, the restriction of the bilinear pairing \eqref{eq:bilinear-form},
\begin{equation*}
(-,-) \colon  \pi^1_x(G) \times \Kbb \pi_1(G,x) \to \Kbb, \quad (\oint_x u, [\alpha]) \mapsto \oint_{\alpha} u,
\end{equation*}
remains non-degenerate.

With the following lemma, we obtain Theorem \ref{thm:pi1-hopf}.

\begin{lemma}[Lemma 3.7, \cite{chen1971algebras}]\label{subal}
Let $\rho\colon H' \to H_1'$ be a morphism of Hopf algebras. If $H \times H' \to \Kbb$ is a left non-degenerate pairing of Hopf algebras, then
$H_1 = (\ker \rho)^{\perp}$ is a Hopf subalgebra of $H$.
\end{lemma}
\begin{theorem}\label{thm:pi1-hopf}
{The algebra $\pi^1_x(G)$ carries a Hopf algebra structure with antipode, whose multiplication is given by the shuffle product, and its unit, comultiplication, counit, and antipode are obtained by restricting $\eta_{\Qcal_x(G)}$, $\Delta_{\Qcal_x(G)}$, $\epsilon_{\Qcal_x(G)}$, and $j_{\Qcal_x(G)}$, respectively.}
	\end{theorem}
\begin{proof}[Proof of Theorem \ref{thm:pi1-hopf}] Let {$\rho$ be the canonical morphism of Hopf algebras $\rho \colon  \Bar{L}_x(G)\to \Kbb\pi_1(G)$ and $N$ be its kernel}. Then the kernel $N$ is spanned by elements of the type $\alpha-\beta$ such that the loops $\alpha$ and $\beta$ are $C_{\partial}$-homotopic. {Since the bilinear pairing $\pi^1_{x}(G)\times \Kbb\pi_{1}(G,x)\to \Kbb$ is {non-degenerate}  and $\pi^1_x(G)=N^{\bot }$}, it follows from Lemma \ref{subal} that $\pi^1_x(G)$ is a Hopf subalgebra of $\Qcal_x(G)$.
\end{proof}

\subsection{Change of the Base Point}
Let $G=(G_0,G_1)$ be a connected digraph, that is, for any two vertices $x,y \in G_0$, there exists a path map $\gamma \colon I_n \to G$ such that $\gamma(0)=x$ and $\gamma(n)=y$. Then there is a well-defined map between the loop spaces $\bar{L}_x(G)$ and $\bar{L}_y(G)$ given by
\begin{equation*}
\begin{aligned}
\gamma_* \colon \bar{L}_x(G) \to \bar{L}_y(G), \quad
\alpha \mapsto \gamma^{-1} \star \alpha \star \gamma .
\end{aligned}
\end{equation*}
A direct computation verifies that
\begin{equation*}
\begin{aligned}
\Delta_{\bar{L}_y(G)}(\gamma_* \alpha)
&= \gamma^{-1} \star \alpha \star \gamma \otimes \gamma^{-1} \star \alpha \star \gamma
= (\gamma_* \otimes \gamma_*) \Delta_{\bar{L}_x(G)}(\alpha), \\
\epsilon_{\bar{L}_y(G)}(\gamma_* \alpha)
&= 1
= \epsilon_{\bar{L}_x(G)}(\alpha), \\
j_{\bar{L}_y(G)}(\gamma_* \alpha)
&= \gamma^{-1} \star \alpha^{-1} \star \gamma
= \gamma_* \bigl( j_{\bar{L}_x(G)}(\alpha) \bigr).
\end{aligned}
\end{equation*}
Thus $\gamma_*$ defines an isomorphism of Hopf $\mathbb{K}$-algebras.

Define a map $\gamma^{*} \colon \Qcal_y(G) \to \Qcal_x(G)$ to be the adjoint of $\gamma_{*}$, namely,
\begin{equation*}
(\gamma^{*}\oint_y\omega_1\cdots\omega_r,\alpha)
=
(\oint_y\omega_1\cdots\omega_r,\gamma_{*}\alpha).
\end{equation*}
This map is well-defined since the pairing $(-,-)$ is non-degenerate. Since $\gamma\gamma^{-1}\sim e$, the inverse of $\gamma^{*}$ is exactly $(\gamma^{-1})^*$, this implies that $\gamma^{*}$ is bijective.
Moreover since the image of $\pi_1(G,x)$ under $\gamma_{*}$ lies exactly in $\pi_1(G,y)$, the image of $\pi^1_y(G)$ by $\gamma^{*}$ lies in the subspace $\pi^1_x(G)$. In other words, $\gamma^{*}$ naturally induces a {bijection}
\begin{equation*}
	\gamma^{*}\colon \pi^1_y(G)\to \pi^1_x(G).
\end{equation*}
{The above discussion leads to the following lemma.}
\begin{lemma}
Let $G$ be a connected digraph, and $\gamma \colon I_n \to G$ a path map such that $\gamma(0)=x$ and $\gamma(n)=y$. Then there exists an isomorphism of Hopf $\mathbb{K}$-algebras
\begin{equation*}
\gamma^* \colon \pi^1_y(G) \to \pi^1_x(G).
\end{equation*}
\end{lemma}

\begin{proof}
It suffices to show that $\gamma^*$ is compatible with both the algebra and coalgebra structures. Since the pairing is nondegenerate, its adjoint $\gamma^{*}$ is both an algebra and a coalgebra morphism, hence a morphism of Hopf $\Kbb$-algebras.
\end{proof}

Let $G=(G_0,G_1)$ and $H=(H_0,H_1)$ be two digraphs, let $x$ be a fixed vertex of $G$, and let $f,g \colon G \to H$ be two digraph maps such that $f\backsimeq g$. Then there exists a line digraph $I_n$ and a digraph map $F \colon G \Box I_n \to H$ such that $F|_{G \Box \{0\}}=f$ and $F|_{G \Box \{n\}}=g$. Let $\gamma$ be the path map from $f(x)$ to $g(x)$ defined by $\gamma(t)=F(x,t)$. We have the following lemma.

\begin{lemma}\label{lem:pi1-homotopy-commutes}
Let $f,g$ and $\gamma$ be defined as above. Then the following diagram commutes:
\begin{equation*}
\xymatrix{
\pi^1_{f(x)}(H) \ar[r]^{f^{*}} & \pi^1_x(G) \\
\pi^1_{g(x)}(H) \ar[u]^{\gamma^{*}} \ar[ur]^{g^{*}}.
}
\end{equation*}
\end{lemma}

\begin{proof}
Since the pairing $(-,-)\colon \pi_x^1(G)\times \pi_1(G,x)\to \Kbb $
is non-degenerate, the Lemma follows from the fact that the following diagram commutes.
	\begin{equation*}
\xymatrix{\pi_1(H,f(x))\ar[d]^{\gamma_{*}}&\pi_1(G,x)\ar[l]^{f_{*}}\ar[dl]^{g_{*}}\\
		\pi_1(H,g(x))
		}
	\end{equation*}
\end{proof}
The following theorem states that the algebra $\pi^1$ is a homotopy invariant.
\begin{theorem}\label{homoinva}
{ If $G\simeq H$,} then as Hopf $\Kbb$-algebras $\pi^1_x(G) \cong \pi^1_y(H)$.
\end{theorem}
\begin{proof}
From the definition of homotopy equivalence of digraphs, there exist line digraphs $I_n$ and $I_m$, together with digraph maps
$f\colon G\to H$, $g\colon H\to G$, $F_G\colon G\Box I_n\to G$, and $F_H\colon H\Box I_m\to H$, such that
$F_G(0)=\mathrm{id}_{G}$, $F_H(0)=\mathrm{id}_{H}$, $F_G(n)=g\circ f$, and $F_H(m)=f\circ g$.
Applying Lemma~\ref{lem:pi1-homotopy-commutes} to $F_G$ and $F_H$, we obtain that both
$(g\circ f)^{*}$ and $(f\circ g)^{*}$ are isomorphisms of Hopf $\mathbb{K}$-algebras.
\end{proof}

The previous result may be reformulated as follows.
\begin{theorem}\label{thm:pi-func}
The assignment
\[
\pi^1 \colon  \underline{\hBDig} \to \underline{\mathcal{H}},
\qquad (G_0, G_1, x) \mapsto \pi^1_x(G),
\]
defines a contravariant functor from the homotopy category of based digraphs to the category of Hopf $\Kbb$-algebras.
\end{theorem}

\end{document}